\documentclass[12pt,reqno]{amsart}
\usepackage{a4wide}

\usepackage{tikz}
\numberwithin{equation}{section}
\usepackage{mathrsfs}
\usepackage{amsfonts}
\usepackage{amsmath}
\allowdisplaybreaks[4]
\usepackage{stmaryrd}
\usepackage{amssymb}
\usepackage{amsthm}
\usepackage{mathrsfs}
\usepackage{url}
\usepackage{amsfonts}
\usepackage{amscd}
\usepackage{indentfirst}
\usepackage{enumerate}
\usepackage{amsmath,amsfonts,amssymb,amsthm}
\usepackage{amsmath,amssymb,amsthm,amscd}
\usepackage{graphicx,mathrsfs}
\usepackage{appendix}
\usepackage{color}

\usepackage[numbers,sort&compress]{natbib}

\newcommand{\R}{{\mathbb R}}

\numberwithin{equation}{section}

\newtheorem{theorem}{Theorem}[section]

\newtheorem{lemma}{Lemma}[section]
\newtheorem{proposition}{Proposition}[section]
\newtheorem{remark}{Remark}[section]

\newcommand{\ds}{\displaystyle}

\begin{document}
\title[]
{\textbf{Infinitely many multi-peaks  solutions for a nonlinear Hartree system }}
%{\textbf{New type of bubble solutions for fractional Schr\"{o}dinger equation involving critical exponent}}
\author{Qihan He~and~Qingfang Wang}

\address[Qihan He]{College of Mathematics and Information Sciences \&
 Center for Applied Mathematics of Guangxi (Guangxi University),  Guangxi University, Guangxi 530003,
People’s Republic of China }
\email{heqihan277@163.com}

\address[Qingfang Wang ]{School of Mathematics and Computer Science, Wuhan Polytechnic University, Wuhan 430079, P. R. China}
\email{wangqingfang@whpu.edu.cn}
\thanks{}

\date{}
\maketitle

\begin{abstract}
In this paper, we study the following nonlinear Hartree system
$$\left\{%
\begin{array}{ll}
-\Delta u_i +V_i(x)u_i=\mu_i\phi_{u_i}u_i+\sum\limits_{j\neq i }\beta_{ij}\phi_{u_j}u_i,~~&x\in \R^3 \vspace{0.12cm}\\
u_i\in H^1( \R^3),i=1,2, 3,\\
\end{array}%
\right.$$
where $ \phi_u(x):=\ds\int_{\R^3}\frac{u^2(y)}{|x-y|}dy$ for any $u\in H^1(\R^3),$  $V_i(x) (i=1,2,3)$ are continuous bounded  radial functions, $\beta_{ij}$ are coupling constants.
We mainly investigate the effects of the potentials and the nonlinear coupling terms on the structure of solutions.
  Applying the Lyapunov-Schmidt reduction method,
we prove the existence of infinitely many solutions
to the system.
Specifically, the solutions we obtain satisfy that some components are synchronized with each other but segregated from the others, and that some components are positive while others are sign-changing. To the best of our knowledge, it is the first time that solutions possessing some positive components and some sign-changing ones have been constructed using Lyapunov-Schmidt reduction methods. Moreover, it is also the first attempt to investigate systems consisting of three Hartree equations with mixed couplings.

 \textbf{Key words: }Hartree system; synchronized components;
segregated solutions; the finite dimensional reduction argument; non-local.\\

\textbf{AMS Subject Classification:} 35J10, 35J20, 34D10.
\end{abstract}

\section{\bf Introduction}\label{s1}
We want to investigate the following nonlinear Hartree system
\begin{equation}\label{1.1}\left\{%
\begin{array}{ll}
-\Delta u_i +V_i(x)u_i=\mu_i\phi_{u_i}u_i+\sum\limits_{j\neq i }\beta_{ij}\phi_{u_j}u_i,~~&x\in \R^3, \vspace{0.12cm}\\
u_i\in H^1( \R^3),i=1,2, 3\\
\end{array}%
\right.\end{equation}
where $ \phi_u(x):=\ds\int_{\R^3}\frac{u^2(y)}{|x-y|}dy$ for any $u\in H^1(\R^3),$  $V_i(x)(i=1, 2, 3)$ are  continuous   radial
functions, and $\beta_{ij}=\beta_{ji}(i\neq j)$ is a coupling constant.
These types of system are  closely related to the following  system
\begin{equation}\label{24.1.1}\left\{%
\begin{array}{ll}
i\partial_t\varphi_i=-\Delta \varphi_i +(V_i+E_i)\varphi_i-\mu_i (K(x)*\varphi_i^2)\varphi_i-\sum\limits_{j\neq i}\beta_{ij} (K(x)*\varphi_j^2)\varphi_i, \vspace{0.12cm}\\
(t,x)\in  \R^+\times\R^3,i=1,2,3
\end{array}%
\right.
\end{equation}
which has important applications in many physical problems, such as in nonlinear optics, the dynamics of boson stars in mean-field theory and the Hartree-Fock theory for a double condensate(See \cite{1es,1eb,1pekar, 1dgps, ms,clll, r}). Physically,  $\varphi_i:(t,x)\in \R^+\times\R^N \to \mathbb{C} (i=1,2,3)$ are the wave functions of the corresponding condensates, $ (V_i+E_i)(i=1,2,3)$  are the external potentials, $K(x)$ denotes the nonnegative response
function, $\mu_i$ is  the strength of the self-interactions in each component, and $\beta_{ij}(i\neq j)$ is  the coupling constant.
The sign of $\beta_{ij}(i\neq j)$ determines whether the interactions of the states are repulsive or attractive. For the attractive case, the components of a vector solution tend to go along with each other, leading to synchronization.
For the repulsive case, the components tend to segregate from each other, leading to phase separations.

It is easy to see that looking for a solution with this type
$$(\varphi_1,\varphi_2,\varphi_3):=(e^{-iE_1t}u_1(x), e^{-iE_2t}u_2(x),e^{-iE_3t}u_3(x))$$
 to \eqref{24.1.1} is equivalent to showing that $(u_1, u_2, u_3)$ is a solution of
\begin{equation}\label{24.1.2}\left\{%
\begin{array}{ll}
 -\Delta u_i + V_i(x)u_i =\mu_i (K(x)*u_i^2)u_i+\sum\limits_{j \neq i}\beta_{ij} (K(x)*u_j^2)u_i, \, x\in \R^N\vspace{0.12cm}\\
  i=1,2,3.
\end{array}%
\right.
\end{equation}
If the response function is the Dirac delta function, i.e, $K(x)=\delta(x)$, then \eqref{24.1.2} can be written as
\begin{equation}\label{24.1.3}\left\{%
\begin{array}{ll}
 -\Delta u_i + V_i(x)u_i =\mu_iu_i^3+\sum\limits_{j \neq i}\beta_{ij} u_j^2u_i,~~\, x\in   \R^N,& \vspace{0.12cm}\\
i=1,2,3,
\end{array}%
\right.
\end{equation}
which is a local problem and has attracted the attention of many researchers over the past decades.
Many results about the existence, multiplicity and properties of the solutions to \eqref{24.1.3} have been obtained. We first introduce some results for the case where some $u_i$ is always zero. For this case,
Problem \eqref{24.1.3} is reduced to   the following problem
\begin{equation}\label{224.1.3}\left\{%
\begin{array}{ll}
 -\Delta u + V_1(x)u =\mu_1u^3+\beta_{12} v^2u,~~&x\in   \R^N, \vspace{0.12cm}\\
-\Delta v + V_2(x)v= \mu_2 v^3+\beta_{12} u^2 v,~~&x\in   \R^N.\\

\end{array}%
\right.
\end{equation}
 Wei-Yao \cite{hwy} and Chen-Zou \cite{CZ}  established the uniqueness of positive solutions for  appropriate parameter $\beta_{12}$ in the case  $N\leq 3$, when  $V_1(x)$ and $V_2(x)$ are suitable constants.
Peng-Wang \cite{PW}  assumed  $N=3$  and that  $V_1(x)$ and $V_2(x)$ are continuous radial functions  satisfying  the following conditions:~ for each fixed  $i=1,2$,   $\inf\limits_{r \geq 0}V_i(r)>0$  and there are  constants $ \theta_i, m_i>0$ and $ a_i \in \R$ such that,
$$
V_i(r)=1+\frac{a_i}{r^{m_i}}+O(\frac{1}{r^{m_i+\theta_i}}),~\hbox{as}~  r\rightarrow +\infty,
$$
and proved that Problem  \eqref{224.1.3} admits  infinitely many synchronized and segregated non-radial  positive solutions for suitable $a_i, m_i, \beta_{12}$ via  the finite dimensional reduction argument.
To construct the synchronized solution, they also showed the existence of non-degenerate positive solutions for the limit system of  \eqref{224.1.3} by means of spectral analysis.
There are also many  results on the system, consisted of three equations, in $\R^3$, such as \cite{lw,PWW,LWW, WY}.
 Lin-Wei \cite{lw} considered the constant potential case  and established  that the solutions are fully segregated, with  each of the three components has a peak moving away from each other  if either  all $\beta_{ij}$ are negative  or one of the $\beta_{ij}$ is negative and the coefficient matrix $(\beta_{ij})$ is definitely positive.
 Peng-Wang-Wang~\cite{PWW}~considered the three-components system in which  the third component exhibits a peak at the origin, while the other two components form synchronized peaks at  the vertices
of a scaled regular polygon far away from the origin. This demonstrates the coexistence of segregation and synchronization phenomena.
Liu-Wang-Wang \cite{LWW} obtained solutions for mixed potentials in which some components are synchronized with each other while being segregated from the others. Wang-Ye \cite{WY} proved the existence of positive solutions that exhibit both synchronization and segregation simultaneously.
For more results on \eqref{24.1.3} and \eqref{224.1.3},  we refer the readers to  \cite{ac1,ac2, afm, HG, hl,hp, lw, ll1,ll2,WW2,LWW1,PV,1lty, 1wangzhou} and the references therein.

When the response function is Riesz potential, i.e, $K(x)=\frac{1}{|x|^\mu }$,
 \eqref{24.1.2} is  transformed into the following system
 \begin{equation}\label{s24.1.2}\left\{%
\begin{array}{ll}
 -\Delta u_i + V_i(x)u_i =\mu_i (\frac{1}{|x|^\mu}*u_i^2)u_i+\sum\limits_{j \neq i}\beta_{ij} (\frac{1}{|x|^\mu}*u_j^2)u_i,~~x\in \R^N, \vspace{0.12cm}\\
  i=1,2,3,
\end{array}%
\right.
\end{equation}
which can be seen as a counterpart of the equation
\begin{equation}\label{ss24.1.2}
 -\Delta u + V(x)u =\phi_uu, x\in \R^3
\end{equation}
and a counterpart of the system
\begin{equation}\label{sss24.1.4}\left\{%
\begin{array}{ll}
-\Delta u +V_1(x)u=\mu_1\phi_uu+\beta_{12} \phi_vu,~~&x\in \R^3, \vspace{0.12cm}\\
-\Delta v +V_2(x)v=\mu_2 \phi_v v+\beta_{12} \phi_uv,~~&x\in \R^3.
\end{array}%
\right.\end{equation}
 As we known, Moroz-Tod \cite{mt} and Wei-Winter \cite{ww} gave some properties of the ground state solution of \eqref{ss24.1.2} with  constant potential. More precisely, they showed that
the following problem
\begin{equation}\label{11.4}
\left\{\begin{array}{ll}
&-\Delta u + u= \phi_u u,~~x\in \R^3, \vspace{0.12cm}\\
&u\in H^1(\R^3),\,\,\,\, u>0, \vspace{0.12cm}\\
&u(0)=\max\limits_{x\in \R^3} u(x),
\end{array}
\right.
\end{equation}
has a unique radial solution, denoted by $W$, and
 $$
 {W^\prime(r)<0},\lim\limits_{|x|\to +\infty}W(x)|x|e^{|x|}=\lambda_0,
\lim\limits_{|x|\to +\infty}\frac{W^\prime(x)}{W(x)}=-1, \lim\limits_{|x|\to +\infty}\phi_W|x|=\lambda_1
$$
 for some constants $\lambda_0, \lambda_1>0$.
Besides, they also proved the non-degeneracy of the positive solution $w$. That is, if $\varphi$ is a solution of the following eigenvalue problem
$$
-\Delta v+ v=\phi_w v+2\int_{\R^3}\frac{Wv}{|x-y|}~dy W,
$$
then
$$
\varphi\in span\Big\{\frac{\partial W}{\partial x_i},~i=1,2, 3\Big\}.
$$
In \cite{ww}, Wei and Winter also constructed multi-bump solutions of \eqref{ss24.1.2} with a small parameter $\epsilon>0$ in front of  the Laplace operator.
Luo-Peng-Wang in \cite{lpw-2020} proved that the multi-bump solutions obtained in \cite{ww} are local uniqueness.
Guo-Luo-Wang-Yang in \cite{glwy-2023} proved the existence and the local uniqueness of a normalized peak solution for \eqref{ss24.1.2} which
concentrates at a non-isolated and degenerate critical point of the potential function. Furthermore, they also  proved
the non-existence of normalized multi-peak solutions for \eqref{ss24.1.2}, which is totally different from the local case.
Hu-Jevnikar-Xie in \cite{hjx} applied the finite dimensional reduction argument to construct infinitely many non-radial positive solutions for \eqref{ss24.1.2}
when $V(x)$ satisfies suitable  algebraic decay. Much more results on \eqref{ss24.1.2} can be seen in \cite{1l, 1lieb, 1lions, 1mz, 1ms}.

Under some assumptions on the parameters,  Wang-Shi \cite{ws}  studied the existence and non-existence of ground state solutions
to  \eqref{sss24.1.4} with positive constant potentials, and also derived qualitative properties of the ground state solutions.
Notably, in \cite{hpwz}, the authors constructed infinitely many both synchronized and segregated solutions by the finite dimensional reduction argument. They first proved that if $(\mu_1, \mu_2, \beta_{12})\in \mathbb{D}$, then   the limit system
\begin{equation}\label{p}
\begin{cases}
-\Delta u+u=\mu_1\phi_uu+\beta_{12} \phi_vu,~x\in \R^3,\\
-\Delta v+v=\mu_2\phi_vv+\beta_{12} \phi_uv,~x\in \R^3,\\
\end{cases}
\end{equation}
has at least a non-degenerate positive solution $(U,V):=(\alpha W, \gamma W)$,
where $\alpha=\sqrt{\frac{\mu_1-\beta_{12}}{\mu_1\mu_2-\beta_{12}^2}},\gamma=\sqrt{\frac{\mu_2-\beta_{12}}{\mu_1\mu_2-\beta_{12}^2}}$
and
\begin{equation}\label{domain}
\mathbb{D}:=\Bigl\{(\mu, \nu, \beta_{12})| \mu, \nu~\hbox{and}~\beta_{12}~\hbox{satisfy}~(I), (II)~\hbox{or}~(III)\Bigr\}
\end{equation}
with

(I)~~~~ $\mu, \nu>0$ and  $\beta_{12}\in (-\sqrt{\mu\nu}, 0)\cup (0,\min\{\mu, \nu\})\cup (\max\{\mu, \nu\}, +\infty) \setminus\{\beta_l\}$;

 (II)~~~~$\mu, \nu<0$ and $\beta_{12}>\sqrt{\mu\nu}$;

 (III)~~~~$\mu\nu\leq 0$ and $\beta_{12}>\max\{\mu, \nu\}$;\\
 and    $\{\beta_l\}\subset(-\sqrt{\max\{\mu\nu,~0\}}, 0) $ is a decreasing sequence with $\beta_l\to -\sqrt{\max\{\mu_1\mu_2,~0\}}$ as $l\to +\infty$.\\
Based on this non-degeneracy result, they established the existence of   synchronized positive solution $\Bigl(\sum\limits_{j=1}^{k} U_{x^j}+\varphi_1(x),  \sum\limits_{j=1}^{k} V_{x^j}+\phi_1(x) \Bigr)$ and  synchronized sign-changing  solution  $\Bigl(\sum\limits_{j=1}^{k} \big(-1\big)^jU_{x^j}+\varphi_2(x),  \sum\limits_{j=1}^{k}\big(-1\big)^j V_{x^j}+\phi_2(x) \Bigr)$ for the system \eqref{sss24.1.4}. At the same time, they also applied the
non-degeneracy results of \cite{mt,ww} and  Lyapunov-Schmidt reduction methods to construct
segregated  positive solution $\Bigl(\sum\limits_{j=1}^{k} W_{1,x^j}+\varphi_3(x),  \sum\limits_{j=1}^{k} W_{2,y^j}+\phi_3(x) \Bigr)$ and segregated sign-changing  solution  $\Bigl(\sum\limits_{j=1}^{k} \big(-1\big)^jW_{1,x^j}+\varphi_4(x),  \sum\limits_{j=1}^{k}\big(-1\big)^j W_{2,y^j}+\phi_4(x) \Bigr)$ for it. Recently, Gao-Yang-Zhao \cite{GYZ2} considered \eqref{1.1} with $V_1(x)=V_2(x)=V_3(x)$ and $\beta_{ij}=\beta$, and constructed non-radial segregated solutions.
We also find  some other results on the more general coupled Hartree system, such as \cite{gdw,gmyz,GYZ1, ywd,1wang} and the references therein.
% In \cite{wy}, Wang-Yang showed the existence and non-existence of normalized solutions to Hartree system \eqref{1.1}, and investigated the asymptotic behavior of the normalized ground state solutions under certain types of  trapping potentials.

 Inspired by \cite{hjx,LWW, WY,GYZ2}, we want to study coexistence of infinitely many synchronized and segregated non-radial   solutions to \eqref{1.1},
where $V_1(x), V_2(x)$ and $V_3(x)$    satisfy the following conditions:

$(A_1)$:~~For any fixed  $i=1,2,3$,  $V_i(x)$ is a continuous radial function, $\inf\limits_{r \geq 0}V_i(r)>0$  and there are constants $ \theta_i, m_i>0, a_i, b_i\in \R$ such that, as $r\rightarrow +\infty$,
$$
V_i(r)=\lambda_i+\frac{a_i}{r^{m_i}}+O(\frac{b_i}{r^{m_i+\theta_i}})
$$

$(A_2)$:~~There exist two non-empty disjoint sets $N_1$ and $N_2$ such that $N_1\cup N_2=\{1,2, 3\}$, $\lambda_i=1, i\in N_1, \lambda_j=\lambda>0$ for $j\in N_2, \min\{m_i|i\in N_1\}=\min\{m_j|j\in N_2\}=m\in [\frac{1}{2}, 1)$. %Without loss of generality, we assume $\lambda_1=\lambda_2=1, \lambda_3=\lambda>0$,~~$\min\{m_1,m_2\}=m_3=:m\in [\frac{1}{2}, 1)$ .

\begin{remark}
The condition $(A_1)$ can  cover  the case of constant potentials.  For some constant  $V_i(x) \equiv \lambda_i$, we may simply set $a_i = b_i = 0$ and choose $m_i \in \mathbb{R}$ to be larger than any $m_j$ associated with nonconstant potentials.
This ensures that the asymptotic expansion in $(A_1)$ holds trivially as
\[
V_i(x) = \lambda_i = \lambda_i +\frac{a_i}{r^{m_i}}+O(\frac{b_i}{r^{m_i+\theta_i}}).
\]
For this reason, we will henceforth assume that $m_i$ is arbitrarily large whenever $V_i(x)$ is a constant.
\end{remark}

Before stating our results, we first introduce some notations.
\begin{equation}\label{bulambda}
\Lambda=\left\{\begin{array}{ll}&\frac{a_3(\alpha^2+\gamma^2)}{a_1\alpha^2\lambda^{\frac{1}{2}}},~~\hbox{for}~m_1<m_2, \\[2mm]
&\frac{a_3(\alpha^2+\gamma^2)}{a_2\gamma^2\lambda^{\frac{1}{2}}},~~\hbox{for}~m_2<m_1,\\[2mm]
&\frac{a_3(\alpha^2+\gamma^2)}{(a_1\alpha^2+a_2\gamma^2)\lambda^{\frac{1}{2}}},~~\hbox{for}~m_1=m_2,\\[2mm]
\end{array}\right.
\end{equation}

\begin{equation}\label{d5}
f_1(d_1)=\left\{\begin{array}{ll}&(1-m)(a_1B_1)^\frac{1}{1-m}\Big(\frac{m}{D_0}\Big)^\frac{m}{1-m},~~\hbox{for}~m_1<m_2,\\
&(1-m)(a_2B_2)^\frac{1}{1-m}\Big(\frac{m}{D_0}\Big)^\frac{m}{1-m},~~\hbox{for}~m_2<m_1,\\
&(1-m)(a_1B_1+a_2B_2)^\frac{1}{1-m}\Big(\frac{m}{D_0}\Big)^\frac{m}{1-m},~~\hbox{for}~m_1=m_2,\\
\end{array}\right.
\end{equation}
and
\begin{equation}\label{d6}
f_2(d_2)=(1-m)(a_3B_3)^\frac{1}{1-m}\Big(\frac{m}{D_1}\Big)^\frac{m}{1-m},
\end{equation}
where
\begin{align}
B_1=\ds\frac{1}{2}\alpha^2\ds\int_{\R^3}w^2~dx,  B_2=\ds\frac{1}{2}\gamma^2\ds\int_{\R^3}w^2~dx,\,\,\, \hbox{and} \,\,\, B_3=\ds\frac{1}{2}\frac{\lambda^\frac{1}{2}}{\mu_3}\ds\int_{\R^3}w^2~dx.
\end{align}
\begin{align}\label{d0}
D_0:=\ds\frac{1}{4}(\mu_1\alpha^4+\mu_2\gamma^4+2\beta_{12}\alpha^2\gamma^2)\frac{C_w}{\pi},D_1=\frac{1}{4}\frac{\lambda}{\mu_3}\frac{C_w}{\pi},
D_2:=\Big(\frac{\beta_{13}\alpha^2+\beta_{23}\gamma^2}{2}\Big)\frac{\lambda^\frac{1}{2}}{\mu_3}\frac{C_w}{\pi}
\end{align}

\begin{theorem}\label{Th1}
Assume that  $(A_1)$, $(A_2)$, $(\mu_1, \mu_2,  \beta_{12})\in \mathbb{D}$ hold and $N_1=\{1,2\}, \mu_3>0$. Then  there exist
$ \beta_0, k_0>0$ such that,  for  any integer $k\geq k_0$,  \eqref{1.1} has a segregated vector positive  solution $(u^1_{1k},~u^1_{2k},~u^1_{3k})$
with $k$  peaks and three segregated vector sign-changing   solution $(u^2_{1k},~u^2_{2k},~u^2_{3k})$,  $(u^3_{1k},~u^3_{2k},~u^3_{3k})$ and  $(u^4_{1k},~u^4_{2k},~u^4_{3k})$
with $k$  peaks,  whose energies depend on $k$ and tend to infinity  as $k \to +\infty$, provided one of the following assumptions holds:

(1)\quad $\Lambda\neq 1,   \beta_{13}\chi(\beta_{13})+\beta_{23}\chi(\beta_{23})<\beta_0 , a_3>0$ and  $a_1>0$ for $m_1<m_2$; $a_2>0$ for $m_2<m_1$; $a_1\alpha^2+a_2\gamma^2>0$ for $m_1=m_2$;\\

(2)\quad $\Lambda=1, \beta_{13}\chi(\beta_{13})+\beta_{23}\chi(\beta_{23})<\beta_0,  -\frac{D_0+D_1}{2}<D_2<\frac{d_1}{2}\min\{f_1(d_1), f_2(d_2)\}, a_3>0$ and  $a_1>0$ for $m_1<m_2$; $a_2>0$ for $m_2<m_1$; $a_1\alpha^2+a_2\gamma^2>0$ for $m_1=m_2$;\\
%
%(3)\quad $\Lambda=1, \beta_{13}\chi(\beta_{13})+\beta_{23}\chi(\beta_{23})<\beta_0, -\frac{D_0+D_1}{2}<D_2<0,   a_3>0$;\\
where $\Lambda, D _0, D _1, D _2, f_1(d_1), f_2(d_2)$ are defined in \eqref{bulambda}-\eqref{d0}, and $\chi(t)=1$ for $t>0$ and $\chi(t)=0$ for $t\leq 0$.
\end{theorem}

Next, we introduce some notations to be used in the proofs of  the main results and formulate
 the versions of the main results, which give more precise
descriptions about segregated character of the solutions. In doing so,
we also outline the main idea and the approaches in the proofs of
Theorems \ref{Th1}.

We first introduce a notation $sgn(\pm)$:
$$sgn(\pm):=\left\{\begin{array}{ll}
1,~~\hbox{for~constructing~positive~solutions};\\
-1,~~\hbox{for~constructing~sign-changing~solutions~when~}~k~\hbox{is~even}.
\end{array}\right.
$$

For any $k\in \mathbb{N}$,
define
\begin{equation}\label{1.2}\begin{split}
H_s=\bigg\{u\in H^1(\R^3): &u~\hbox{is~even~in~}x_h, h=2,3,u\Big(r\cos\big(\theta+\frac{2\pi j}{k}\big), r\sin\big(\theta+\frac{2\pi j}{k}\big),x_3\Big)
\\
&\quad \quad =(sgn(\pm))^ju\left(r\cos\theta, r\sin\theta,x_3\right),j=1,2,\cdots, k\bigg\},
\end{split}
\end{equation}
where $H^1(\R^3)$ is the usual Sobolev space with the norm for any bounded function $V(x)$
$$\|u\|^2_{V}=(u,u)=\int_{\R^3}(|\nabla u|^2 +V(x)|u|^2)dx,$$
and define  $ H  = H_s \times H_s \times H_s$   endowed with the following norm
$$\|(u_1,u_2,u_3)\|^2=\|u_1\|^2_{V_1}+\|u_2\|^2_{V_2}+\|u_3\|^2_{V_3} . $$

For any $u\in H^1(\R^3)$ and $y\in \R^3$, we set $u_{y}(x)=u\big(x-y\big)$.
Let $(U,V)$ be the non-degenerate solution,  given in Theorem 1.1 of \cite{hpwz}, to \eqref{p},  and define
\begin{align}
U_r(x)=\sum\limits_{j=1}^{k}(sgn(\pm))^j U_{x^j},\,\,\,V_r(x)=\sum\limits_{j=1}^{k}(sgn(\pm))^j V_{x^j},\,\,W_\rho=\sum\limits_{j=1}^{k}(sgn(\pm))^j W_{y^j}
\end{align}
where $W(x) = \frac{\lambda}{\sqrt{\mu_3}} \, w\bigl( \sqrt{\lambda} \, x \bigr)$,
\begin{equation}\label{1.3}
 x^j:=\Big(r\cos\frac{2(j-1)\pi}{k},~r\sin\frac{2(j-1)\pi}{k},~x_3\Big),~j=1,2,\cdots,k,~~r \in S_k,
\end{equation}

\begin{equation}\label{1.4}
y^j:=\Big(\rho\cos\frac{(2j-1)\pi}{k},~\rho\sin\frac{(2j-1)\pi}{k},~x_3\Big),~j=1,2,\cdots,k,~~\rho\in S_k,
\end{equation}
and
\begin{equation}\label{s} S_k:=\Big[C_1 (k\ln k)^\frac{1}{1-\min\{m,n\}},\quad C_2(k\ln k)^\frac{1}{1-\min\{m,n\}}\Big],
\end{equation}
%\begin{equation}\label{bus} \tilde{S} :=\Big[(d_k^\frac{1}{m+1}-\tilde{\delta}) \epsilon^\frac{1}{m+1},
%\quad (d_k^\frac{1}{m+1}+\tilde{\delta})\epsilon^\frac{1}{m+1}\Big],
%\end{equation}
here $C_1, C_2 $ are two constants independent of $k$, which will be  determined later  and whose values may vary from place to place.

%It is well-known that the following problem
%\begin{equation}\label{11.4}
%-\Delta u+u =\phi_u u,~ \max\limits_{x\in\R^3}u(x)=u(0),u>0,
%\end{equation}
% has a unique radial solution denoted by $w$
%and the solution $w$ satisfies the following properties:
%$$
%w^\prime(r)<0,\,\,
%    \lim\limits_{r\rightarrow\infty}re^rw(r)=C_0>0,\,\,
% \lim\limits_{r\rightarrow\infty}\frac{w^\prime(r)}{w(r)}=-1~\hbox{and}~ \lim\limits_{r\rightarrow\infty}\phi_w(r)r=C_1>0.
%$$

%When $-\sqrt{\mu_1\mu_2}<\beta<\min\{\mu_1,\mu_2\}$ or $\beta>\max\{\mu_1,\mu_2\}$,
%$(U,V):=(\alpha w, \gamma w)$ is a solution of
%the following system:
%\begin{equation}\label{sys}\left\{%
%\begin{array}{ll}
%-\Delta u +u =\mu_1\phi_uu +\beta \phi_v u,\,\,&x\in \R^{3},\vspace{0.15cm}\\
%-\Delta v +v =\mu_2\phi_v v  +\beta \phi_u v,\,\,&x\in \R^{3},\\
%\end{array}
%\right.
%\end{equation}
%where $\alpha=\sqrt{\frac{\mu_2-\beta}{\mu_1\mu_2-\beta^2}},\gamma=\sqrt{\frac{\mu_1-\beta}{\mu_1\mu_2-\beta^2}}.$

We will verify Theorem \ref{Th1} by proving the following result:
\begin{theorem}\label{Th3}
~Under the assumptions of Theorem \ref{Th1},  Problem \eqref{1.1} has a segregated vector positive  solution
$$
\Bigl(\sum\limits_{j=1}^{k}  U_{x^j}+\varphi_1(x),\,\,
\sum\limits_{j=1}^{k}  V_{x^j}+\psi_1(x),\,\,\sum\limits_{j=1}^{k}  W_{y^j}+\xi_1(x)\Bigr)
$$
and three  segregated vector sigh-changing   solutions
$$
\Bigl(\sum\limits_{j=1}^{k} (-1)^j U_{x^j}+\varphi_2(x),\,\,
\sum\limits_{j=1}^{k}  (-1)^j V_{x^j}+\psi_2(x),\,\,\sum\limits_{j=1}^{k} (-1)^j  W_{y^j}+\xi_2(x)\Bigr),
$$
$$
\Bigl(\sum\limits_{j=1}^{k} U_{x^j}+\varphi_3(x),\,\,
\sum\limits_{j=1}^{k} V_{x^j}+\psi_3(x),\,\,\sum\limits_{j=1}^{k} (-1)^j  W_{y^j}+\xi_3(x)\Bigr)
 $$
and
$$
\Bigl(\sum\limits_{j=1}^{k} (-1)^j U_{x^j}+\varphi_4(x),\,\,
\sum\limits_{j=1}^{k}  (-1)^j V_{x^j}+\psi_4(x),\,\,\sum\limits_{j=1}^{k} W_{y^j}+\xi_4(x)\Bigr).
$$

Moreover,
$$
||( \varphi_i(x),
 \psi_i(x), \xi_i(x))||= O\Big(\Big(\ds\frac{1}{k}\Big)^\frac{2m-1}{1-m}\Big(\ds\frac{1}{\ln k}\Big)^\frac{(1-\tau_2)m}{1-m}\Big), ~~i=1,2,3,4.
 $$
\end{theorem}

\begin{remark}
When considering sign-changing solutions, k is necessarily even.
\end{remark}

\begin{remark}
Replacing $W_\rho$ and $\xi$ by $0$ and $0$, and repeating a similar process, one finds that  Problem \eqref{1.1}  admits a segregated vector positive  solution
$$
\Bigl(\sum\limits_{j=1}^{k}  U_{x^j}+\varphi_1(x),\,\,
\sum\limits_{j=1}^{k}  V_{x^j}+\psi_1(x),\,\, 0\Bigr)
$$
 and a segregated vector sign-changing   solution
$$
\Bigl(\sum\limits_{j=1}^{k} (-1)^j U_{x^j}+\varphi_2(x),\,\,
\sum\limits_{j=1}^{k}  (-1)^j V_{x^j}+\psi_2(x),\,\,0\Bigr),
$$
which has been proven in  \cite{hpwz}. Similarly, letting $V_r\equiv 0$ and $\psi\equiv0$, and repeating a similar argument, one obtains that system \eqref{1.1} has two segregated vector sign-changing solutions with the following form
$$
\Bigl(\sum\limits_{j=1}^{k} U_{x^j}+\varphi_3(x),\,\,
0,\,\,\sum\limits_{j=1}^{k} (-1)^j  W_{y^j}+\xi_3(x)\Bigr),\,\,\,\,\Bigl(\sum\limits_{j=1}^{k} (-1)^j U_{x^j}+\varphi_4(x),\,\,
0,\,\,\sum\limits_{j=1}^{k} W_{y^j}+\xi_4(x)\Bigr),
$$
which implies that  Problem  \eqref{sss24.1.4} also admits such solutions $
\Bigl(\sum\limits_{j=1}^{k} U_{x^j}+\varphi_3(x),\,\,
 \sum\limits_{j=1}^{k} (-1)^j  W_{y^j}+\xi_3(x)\Bigr),\,\,\Bigl(\sum\limits_{j=1}^{k} (-1)^j U_{x^j}+\varphi_4(x),\,\,
 \sum\limits_{j=1}^{k} W_{y^j}+\xi_4(x)\Bigr),
$
under the assumptions of Theorem 1.3 in \cite{hpwz}.
\end{remark}

\begin{remark}By symmetry,  if conditions $(A_1), (A_2)$,  $(\mu_2,\mu_3,\beta_{23})\in \mathbb{D}$ hold and $N_1=\{2,3\}, \mu_1>0$, then    there exist
$ \beta_0, k_0>0$ such that,  for  any integer $k\geq k_0$,  \eqref{1.1}  admits a segregated vector positive  solution
$$
\Bigl(\sum\limits_{j=1}^{k}  W_{y^j}+\xi_1(x),\,\, \sum\limits_{j=1}^{k}  U_{x^j}+\varphi_1(x),\,\,
\sum\limits_{j=1}^{k}  V_{x^j}+\psi_1(x)\Bigr)
$$
and three  segregated vector sign-changing   solutions
$$
\Bigl(\sum\limits_{j=1}^{k} (-1)^j  W_{y^j}+\xi_2(x),\,\, \sum\limits_{j=1}^{k} (-1)^j U_{x^j}+\varphi_2(x),\,\,
\sum\limits_{j=1}^{k}  (-1)^j V_{x^j}+\psi_2(x)\Bigr),
$$
$$
\Bigl(\sum\limits_{j=1}^{k} (-1)^j  W_{y^j}+\xi_3(x),\,\,\sum\limits_{j=1}^{k} U_{x^j}+\varphi_3(x), \,\,
\sum\limits_{j=1}^{k} V_{x^j}+\psi_3(x)\Bigr)
 $$
and
$$
\Bigl(\sum\limits_{j=1}^{k} W_{y^j}+\xi_4(x),\,\,\sum\limits_{j=1}^{k} (-1)^j U_{x^j}+\varphi_4(x),\,\,
\sum\limits_{j=1}^{k}  (-1)^j V_{x^j}+\psi_4(x)\Bigr),
$$
  provided one of the following assumptions holds:

(1)\quad $\Lambda\neq 1,   \beta_{21}\chi(\beta_{21})+\beta_{31}\chi(\beta_{31})<\beta_0 , a_1>0$ and  $a_2>0$ for $m_2<m_3$; $a_3>0$ for $m_3<m_2$; $a_2\alpha^2+a_3\gamma^2>0$ for $m_2=m_3$;\\

(2)\quad $\Lambda=1,  \beta_{21}\chi(\beta_{21})+\beta_{31}\chi(\beta_{31})<\beta_0,  -\frac{D_0+D_1}{2}<D_2<\frac{d_1}{2}\min\{f_1(d_1), f_2(d_2)\}, a_1>0$ and  $a_2>0$ for $m_2<m_3$; $a_3>0$ for $m_3<m_2$; $a_2\alpha^2+a_3\gamma^2>0$ for $m_2=m_3$;\\
%
%(3)\quad $\Lambda=1, \beta_{13}\chi(\beta_{13})+\beta_{23}\chi(\beta_{23})<\beta_0, -\frac{D_0+D_1}{2}<D_2<0,   a_3>0$;\\
where $\Lambda, D _0, D _1, D _2, f_1(d_1), f_2(d_2)$,    defined in \eqref{d5}-\eqref{d0},  are slightly adjusted in their definitions by symmetry  and $\chi(t)=1$ for $t>0$ and $\chi(t)=0$ for $t\leq 0$.

Similar results hold when  $(A_1), (A_2)$,  $(\mu_1,\mu_3,\beta_{13})\in \mathbb{D}$ hold and $N_1=\{1,3\}, \mu_2>0$.
\end{remark}
The solution $(u_1,u_2,u_3)$ given in Theorem 1.1 indicates that $u_1$ and $u_2$ are synchronized, while they are segregated from $u_3$. The second picture indicates that $u_1$ and $u_3$ are synchronized, while they are segregated from $u_2$.

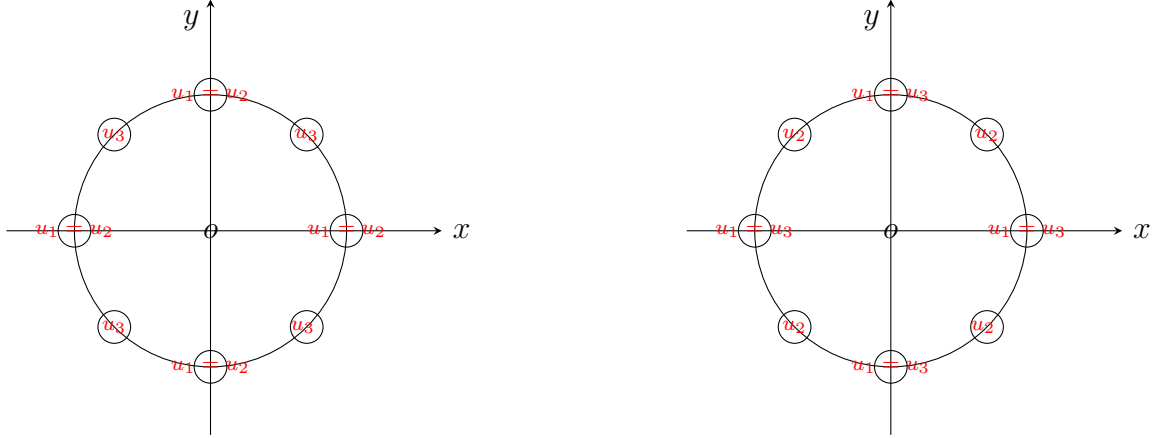
\begin{tikzpicture}[>=stealth, scale=1.8]
    % 第一个图（左边）
    \begin{scope}
        % 坐标轴
        \draw[->] (-1.5, 0) -- (1.7, 0) node[right] {$x$};
        \draw[->] (0, -1.5) -- (0, 1.7) node[below left] {$y$};
        % 圆心
        \node at (0, 0) {$o$};
        % 外圆
        \draw (0,0) circle (1);
        % 8个点
        \foreach \angle/\label in {
            90/{$u_1=u_2$},
            45/{$u_3$},
            0/{$u_1=u_2$},
            -45/{$u_3$},
            -90/{$u_1=u_2$},
            -135/{$u_3$},
            180/{$u_1=u_2$},
            135/{$u_3$}
        } {
            \pgfmathsetmacro{\x}{cos(\angle)}
            \pgfmathsetmacro{\y}{sin(\angle)}
            \draw (\y, \x) circle (0.12);
            \node[red, font=\tiny, align=center] at (\y, \x) {\label};
        }
    \end{scope}

    % 第二个图（右边，整体向右平移3个单位）
    \begin{scope}[xshift=5cm]
        % 坐标轴
        \draw[->] (-1.5, 0) -- (1.7, 0) node[right] {$x$};
        \draw[->] (0, -1.5) -- (0, 1.7) node[below left] {$y$};
        % 圆心
        \node at (0, 0) {$o$};
        % 外圆
        \draw (0,0) circle (1);
        % 8个点
        \foreach \angle/\label in {
            90/{$u_1=u_3$},
            45/{$u_2$},
            0/{$u_1=u_3$},
            -45/{$u_2$},
            -90/{$u_1=u_3$},
            -135/{$u_2$},
            180/{$u_1=u_3$},
            135/{$u_2$}
        } {
            \pgfmathsetmacro{\x}{cos(\angle)}
            \pgfmathsetmacro{\y}{sin(\angle)}
            \draw (\y, \x) circle (0.12);
            \node[red, font=\tiny, align=center] at (\y, \x) {\label};
        }
    \end{scope}
\end{tikzpicture}

\section{\bf Finite-dimensional reduction}\label{s2}
\def\theequation{3.\arabic{equation}}\makeatother
\setcounter{equation}{0}
In this section, we shall construct segregated  vector solutions.
The functional corresponding to \eqref{1.1} is
\begin{align}\label{2.1}
I(u_1,u_2,u_3)=\ds\frac{1}{2}\sum\limits_{j=1}^3\ds\int_{\R^3}(|\nabla u_j|^2 +V_j(x)u^2_j)~dx -\frac{1}{4}\sum\limits_{j=1}^3\mu_{j}\ds\int_{\R^3}\phi_{u_j}u_j^2~dx-\frac{1}{4}\sum\limits_{i\neq j}\beta_{ij}\ds\int_{\R^3}\phi_{u_j}u_i^2~dx,
\end{align}
which is of  $C^2(H)$ and thus its critical points correspond to  the solutions of \eqref{1.1}.
Define
\begin{align}
X_j:=\frac{\partial U_{x^j}}{\partial r}, Y_j:=\frac{\partial V_{x^j}}{\partial r}, Z_j:=\frac{\partial W_{y^j}}{\partial \rho},j=1,2,\cdots,k,
\end{align}
and
\begin{align}\label{3.2}
E=\Big\{(\varphi, \psi, \xi)\in H:\sum\limits_{j=1}^{k}(sgn(\pm))^j\langle (X_j, Y_j), (\varphi,\psi) \rangle=0, \sum\limits_{j=1}^{k}(sgn(\pm))^j\langle Z_j, \xi \rangle=0\Big\},
\end{align}
where $x^j$ and $y^j$ are defined in \eqref{1.3}, \eqref{1.4} respectively. Then ${E}$ is a closed subspace of $H$.
Let
$$
J(\varphi,\psi, \xi)=I(U_r+\varphi,V_r +\psi, W_\rho+\xi),\,\,\,(\varphi,\psi,\xi)
\in E.
$$
Then, $J(\varphi,\psi, \xi)$ has the following expansion:
\begin{align}
J(\varphi,\psi,\xi)=J(0,0,0)+\ell(\varphi,\psi,\xi)+\frac{1}{2}Q
(\varphi,\psi,\xi)+R(\varphi,\psi,\xi),\,\,(\varphi,\psi,\xi)\in E,
\end{align}
where
\begin{align}\label{wqs1}
\ell(\varphi,\psi,\xi)
=&\ds\int_{\R^3}(V_1(x)-1)U_r\varphi
-\mu_{1}\ds\int_{\R^3}\Big(\phi_{U_r}U_r -\sum\limits_{j=1}^{k}\phi_{U_{x^j }}U_{x^j }\Big)\varphi\cr
&-\beta_{12}\ds\int_{\R^3} (\phi_{V_r}U_r -\sum\limits_{j=1}^{k}\phi_{V_{x^j }}U_{x^j } )\varphi
-\beta_{13}\ds\int_{\R^3} (\phi_{W_\rho}U_r -\sum\limits_{j=1}^{k}\phi_{W_{y^j }}U_{x^j } )\varphi\cr
&+\ds\int_{\R^3}(V_2(x)-1)V_r\psi
-\mu_{2}\ds\int_{\R^3}\Big(\phi_{V_r}V_r -\sum\limits_{j=1}^{k}\phi_{V_{x^j }}V_{x^j }\Big)\psi\cr
&-\beta_{23}\ds\int_{\R^3}\Big(\phi_{W_\rho}V_r -\sum\limits_{j=1}^{k}\phi_{W_{y^j }}V_{x^j }\Big)\psi
-\beta_{12}\ds\int_{\R^3}\Big(\phi_{U_r}V_r -\sum\limits_{j=1}^{k}\phi_{U_{x^j }}V_{x^j }\Big)\psi\cr
&+\ds\int_{\R^3}(V_3(x)-\lambda)W_\rho\xi
-\beta_{13}\ds\int_{\R^3}\Big(\phi_{U_r}W_\rho -\sum\limits_{j=1}^{k}\phi_{U_{x^j }}W_{y^j }\Big)\xi\cr
&-\beta_{23}\ds\int_{\R^3}\Big(\phi_{V_r} W_\rho -\sum\limits_{j=1}^{k}\phi_{V_{x^j }}W_{y^j }\Big)\xi-\mu_{3}\ds\int_{\R^3}\Big(\phi_{W_\rho}W_\rho -\sum\limits_{j=1}^{k}\phi_{W_{y^j }}W_{y^j }\Big)\xi
\end{align}

\begin{align}
Q(\varphi,\psi,\xi)
&=\ds\int_{\R^3}(|\nabla \varphi|^2+V_1(x)|\varphi|^2-\ds \mu_{1}(\phi_{U_r}|\varphi|^2 +2\phi[U_r\varphi]U_r\varphi)\cr
&+\int_{\R^3}|\nabla \psi|^2+V_2(x)|\psi|^2 -\ds \mu_{2}(\phi_{V_r}|\psi|^2 +2\phi[V_r\psi]V_r\psi)\cr
&+\int_{\R^3}|\nabla \xi|^2+V_3(x)|\xi|^2 )-\ds \mu_{3}(\phi_{W_\rho}|\xi|^2 +2\phi[W_\rho\xi]W_\rho\xi)\cr
&+\beta_{12}\int_{\R^3}\phi_{V_r}|\varphi|^2+2\phi[U_r\varphi]V_r\psi+\phi_{U_r}|\psi|^2+2\phi[U_r\varphi]V_r\psi\cr
&+\beta_{13}\int_{\R^3}\phi_{W_\rho}|\varphi|^2+2\phi[U_r\varphi]W_\rho\xi)+\phi_{U_r}|\xi|^2+2\phi[U_r\varphi]W_\rho\xi\cr
&+\beta_{23}\int_{\R^3}\phi_{W_\rho}|\psi|^2+2\phi[V_r\psi]W_\rho\xi)+\phi_{V_r}|\xi|^2+2\phi[V_r\psi]W_\rho\xi)\cr
\end{align}
and
\begin{align}
 &R(\varphi,\psi,\xi)\cr
=&-\ds\int_{\R^3}(\frac{\mu_{1}}{4}\phi_{\varphi}|\varphi|^2+\mu_{1}\phi_{\varphi}U_r\varphi
+\frac{\mu_{2}}{4}\phi_{\psi}|\psi|^2+\mu_{2}\phi_{\psi}V_r\psi
+ \frac{\mu_{3}}{4}\phi_{\xi}|\xi|^2+\mu_{3}\phi_{\xi}W_\rho\xi)\cr
&-\frac{\beta_{12}}{2}\ds\int_{\R^3} (\phi_{\varphi}|\psi|^2+2\phi_{\psi}U_r\varphi
+2\phi_{\varphi}V_r\psi)
-\frac{\beta_{13}}{2}\ds\int_{\R^3} (\phi_{\varphi}|\xi|^2+2\phi_{\xi}U_r\varphi
+2\phi_{\varphi}W_\rho\xi)\cr
&-\frac{\beta_{23}}{2}\ds\int_{\R^3} (\phi_{\psi}|\xi|^2+2\phi_{\xi}V_r\psi
+2\phi_{\psi}W_\rho\xi)\cr
\end{align}

Consider the following bi-linear functional $B:E\times E \to \R$ given by
\begin{align}
B((u,v,w), (\varphi,\psi,\xi))
=&\ds\int_{\R^3}(\nabla u \nabla \varphi +V_1(x) u \varphi- \mu_{1}(\phi_{U_r}u \varphi +2\phi[U_ru]U_r\varphi) \cr
&+ \ds\int_{\R^3}\nabla v \nabla \psi +V_2(x)v \psi- \mu_{2}(\phi_{V_r}v \psi  +2\phi[V_r v]V_r\psi)\cr
&+ \ds\int_{\R^3}\nabla w\nabla \xi +V_3(x) w \xi )-\mu_{3}(\phi_{W_\rho} w \xi  +2\phi[W_\rho w]W_\rho\xi)\cr
&+\beta_{12}\ds\int_{\R^3}\phi_{V_r} u\varphi +2\phi[U_r u]V_r\psi+\phi_{U_r}v \psi +2\phi[U_r\varphi]V_rv\cr
&+\beta_{13}\ds\int_{\R^3}\phi_{W_\rho}u \varphi +2\phi[U_r u ]W_\rho\xi)+\phi_{U_r}w \xi +2\phi[U_r\varphi]W_\rho w
\cr
&+\beta_{23}\ds\int_{\R^3}\phi_{W_\rho}v \psi +2\phi[V_rv]W_\rho\xi
+\phi_{V_r}w \xi +2\beta_{23}\phi[V_r\psi]W_\rho w
\end{align}
Since $U_r, V_r, W_\rho$ are bounded functions, it is easy to see that
$$|B((u,v,w), (\varphi,\psi,\xi))|\leq C||(u,v,w)||  || (\varphi,\psi,\xi)||,$$
which implies that $B$ is a bounded bi-linear functional in $E$. By the Riesz presentation theorem,  there exists a bounded linear operator $B$ from $E$ to $E$ such that
\begin{align}
B((u,v,w), (\varphi,\psi,\xi))=\langle B(u,v,w), (\varphi,\psi,\xi)\rangle, \forall (u,v,w), (\varphi,\psi,\xi)\in E.
\end{align}
\begin{proposition}\label{pro3.4}
For $k$ sufficiently large, there exists a $C^1$-map $(\varphi,\psi,\xi)$ from $S_k\times S_k\times S_k$ to $H$:$(\varphi,\psi,\xi)=(\varphi(r,\rho),\psi(r,\rho), \xi(r, \rho)), r=|x^i|, \rho=|y^j|,$ satisfying
$(\varphi,\psi, \xi)\in E,$ and
$$
\Big\langle\frac{\partial J(\varphi,\psi,\xi)}{\partial (\varphi,\psi,\xi)},(g,h,f)\Big\rangle=0,\,\,\,\forall (g,h,f)\in E.
$$
Moreover,  there exists a small constant $0<\tau_2< 1$   such that
$$
\begin{array}{rl}
&\|(\varphi,\psi,\xi)\|\leq \Big(\ds\frac{1}{k}\Big)^\frac{2m-1}{1-m}\Big(\frac{1}{\ln k}\Big)^\frac{(1-\tau_2)m}{1-m}.
\end{array}
$$
\end{proposition}
From the above analysis, we have the following lemma:
\begin{lemma}\label{lm3.1}
There exists a constant $C>0$, independent of $k$, such that for any $(r, \rho) \in S_k\times S_k  $
$$\|B(\varphi,\psi,\xi)\| \leq C \|(\varphi,\psi,\xi)\|,\,\,\,\,(\varphi,\psi,\xi)\in E.$$

\end{lemma}
\begin{proof}
The proof is standard, so we omit it.
\end{proof}
In order to apply the contraction mapping theorem to look for $(\varphi,\psi,\xi)$, first we need to prove the invertibility of the operator.
\begin{lemma}\label{lm3.2}Assume that the assumptions of Theorem \ref{Th3} hold. %保证正解的非退化性
There exist $k_0 >0,\beta_0>0$ and $C_0>0$ such that for any $\beta_{13}\chi(\beta_{13})+\beta_{23}\chi(\beta_{23})<\beta_0$ and  any $k>k_0, (r, \rho)\in S_k \times S_k,$
we have
$$\|{B}(\varphi,\psi,\xi)\| \geq C_0 \|(\varphi,\psi, \xi)\|,\,\,\,\,(\varphi,\psi,\xi)\in E.$$
\end{lemma}

\begin{proof}
 We argue by contradiction. Suppose that there are $k_n \to +\infty,  (r_n,\rho_n) \in S_{k_n}\times S_{k_n}$
and $(\varphi_n,\psi_n,\xi_n) \in E$ with $\|(\varphi_n,\psi_n,\xi_n)\|^2=k_n$ satisfying
\begin{equation}\label{3.3}
\langle{B}(\varphi_n,\psi_n,\xi_n),(g,h,f)\rangle=o_n(1)\|(\varphi_n,\psi_n,\xi_n)\|\|(g,h,f)\|,\,\,\,\forall (g,h,f)\in E.
\end{equation}
For $j=1,2,\cdots, k_n$, we set
$$
\Omega_j:=\Big\{z=(z^\prime, z_3)\in \R^2\times\R:~\Big\langle\frac{z^\prime}{|z^\prime|}, \frac{x^{\prime j}}{|x^{\prime j}|}\Big\rangle\geq \cos \frac{\pi}{k_n}\Big\}
$$
and
$$
\tilde{\Omega}_j:=\Big\{z=(z^\prime, z_3)\in \R^2\times\R:~\Big\langle \frac{z^\prime}{|z^\prime|}, \frac{y^{\prime j}}{|y^{\prime j}|}\Big\rangle \geq \cos \frac{\pi}{k_n}\Big\}.
$$
By the symmetry, we get that
$$\|(\varphi_n,\psi_n,\xi_n)\|^2_{H^1(\Omega_j)}=\|(\varphi_n,\psi_n,\xi_n)\|^2_{H^1(\tilde{\Omega}_j)}=1$$
and
\begin{align}\label{3.4}
 &\frac{1}{k_n}\langle B(\varphi_n,\psi_n,\xi_n),(g,h,f)\rangle\cr
=&\ds\int_{\Omega_1}(\nabla \varphi_n \nabla g +V_1(x) \varphi_n g -\ds\mu_{1} (\phi_{U_r}\varphi_n g +2\phi[U_r\varphi_n]U_rg \cr
&+ \int_{\Omega_1}\nabla  \psi_n \nabla h +V_2(x) \psi_n h-\ds\mu_{2} (\phi_{V_r} \psi_n h  +2\phi[V_r  \psi_n]V_rh \cr
&+\int_{\Omega_1} \nabla \xi_n \nabla f +V_3(x) \xi_n f)-\ds\mu_{3}(\phi_{W_\rho} \xi_n f  +2\phi[W_\rho \xi_n]W_\rho f
\cr
&+\beta_{12}\ds\int_{\Omega_1}\phi_{V_r} \varphi_ng +2\phi[U_r \varphi_n]V_rh+\phi_{U_r} \psi_n h +2\phi[U_rg]V_r \psi_n\cr
&+\beta_{13}\ds\int_{\Omega_1}\phi_{W_\rho}\varphi_n  g +2\phi[U_r \varphi_n ]W_\rho f +\phi_{U_r} \xi_n f +2\phi[U_rg]W_\rho \xi_n\cr
&
+\beta_{23}\ds\int_{\Omega_1}\phi_{W_\rho} \psi_n h +2\phi[V_r \psi_n]W_\rho f
+\phi_{V_r}\xi_n f +2\phi[V_rh]W_\rho \xi_n)\cr
&=o(\frac{||(g,h,f)||}{\sqrt{k_n}}),
\end{align}
which is  also true if $\Omega_1$ is replaced by $\tilde{\Omega}_1.$ Given any $R>0$,  we can see that $B_R(x^j)\subset \Omega_j$ and $B_R(y^j)\subset \tilde{\Omega}_j$ for $k_n$ large enough since
$r\sim (k_n\ln k_n)^\frac{1}{1-m}, \rho\sim (k_n\ln k_n)^\frac{1}{1-m}.$

Letting
$$\tilde{u}_n(x)=\varphi_n( x+x^1),\,\,\,\tilde{v}_n(x)=\psi_n(x+x^1), \tilde{\xi}_n(x)=\xi_n(x+y^1),$$ then
$$||\tilde{u}_n(x)||_{H^1(B_R(0))}+||\tilde{v}_n(x)||_{H^1(B_R(0))}+||\tilde{\xi}_n(x)||_{H^1(B_R(0))}\leq 1.$$

%In particular, we have
%\begin{equation}\label{2.6}
%\begin{array}{rl}
%&\ds\int_{\Omega_1}(|\nabla \varphi_n|^2+P(x)\varphi_n^2-\mu_1\phi_{\tilde{U}_{r_n}}\varphi_n^2 -2\mu_1\phi[\tilde{U}_{r_n} \varphi_n] \tilde{U}_{r_n}\varphi_n-\beta\phi_{\tilde{V}_{\rho_n} }\varphi_n^2 -2\beta\phi[\tilde{U}_{r_n}\varphi_n]\tilde{V}_{\rho_n} \psi_n)\vspace{0.12cm}\\
%&+\ds\int_{\Omega_1}(|\nabla \psi_n|^2+Q(x)\psi_n^2-\mu_2\phi_{\tilde{V}_{\rho_n} }\psi_n^2 -2\mu_2\phi[\tilde{V}_{\rho_n} \psi_n]\tilde{V}_{\rho_n} \psi_n-\beta\phi_{\tilde{U}_{r_n}}\psi_n^2
%-2\beta\phi[\tilde{U}_{r_n}\varphi_n] \tilde{V}_{\rho_n} \psi_n)\vspace{0.12cm}\\
%&=o_n(1)
%\end{array}
%\end{equation}
%and
%\begin{equation}\label{bu.2.6.3}
%\ds\int_{\Omega_1}(|\nabla\varphi_n|^2 +P(x)|\varphi_n|^2+|\nabla\psi_n|^2 +Q(x)\psi_n^2)=1.
%\end{equation}
%It is obvious that \eqref{3.4}, \eqref{2.6} and \eqref{bu.2.6.3} are
%
%We set $\tilde{u}_n(x)=\varphi_n( x+x^1),\,\,\,\tilde{v}_n(x)=\psi_n(x+y^1)$.  We mainly consider  $\tilde{u}_n(x)$ and a similar way can be used to $\tilde{v}_n(x)$.
%According to \eqref{bu.2.6.3}, we may assume that
Up to a subsequence, there exist $u, v, w\in H^1(\R^3)$ such that as $n \to +\infty$
\begin{align}
(\tilde{u}_n(x), \tilde{v}_n(x), \tilde{\xi}_n(x)) \to (u, v, w)~~~\hbox{weakly in}~(H^1_{loc}(\R^3))^3,
\end{align}
and
\begin{align}
(\tilde{u}_n(x), \tilde{v}_n(x), \tilde{\xi}_n(x)) \to (u, v, w)~~~\hbox{strongly in}~(L^2_{loc}(\R^3))^3.
\end{align}
Since $(\varphi_n,\psi_n,\xi_n)\in {E}$, it follows from the definition of ${E}$,  and the definitions and weak  convergence of  $\tilde{u}_n(x), \tilde{v}_n(x), \tilde{\xi}_n(x)$  that
$u, v, w$ are even in $x_2, x_3$ and satisfy
\begin{equation}\label{11.20.2}
\Bigl\langle(\frac{\partial U}{\partial x_1}, \frac{\partial V}{\partial x_1} ), (u,v)\Bigr\rangle=0~~\hbox{and}~\Bigl\langle\frac{\partial W}{\partial x_1}, w\Bigr\rangle=0.
\end{equation}

Now we claim that $(u,v)$ satisfies
\begin{equation}\label{2.9}\left\{%
\begin{array}{ll}
-\Delta u +u -\mu_{1}\phi_Uu-2\mu_{1}\phi[Uu]U-\beta_{12} \phi_Vu-2\beta_{12}\phi[Vv]U=0,\,\,&x\in \R^{3},\vspace{0.2cm}\\
-\Delta v +v -\mu_{2}\phi_Vv-2\mu_{2}\phi[Vv]V-\beta_{12} \phi_U v-2\beta_{12} \phi[Uu]V=0, \,\,&x\in \R^{3}.
\end{array}
\right.
\end{equation}
and
$w$ satisfies
\begin{equation}\label{11.20.1}
-\Delta u +u -\mu_{3}\phi_Wu-2\mu_{3}\phi[Wu]W=0.
\end{equation}

Define
$$
\widehat{E}=\left\{ (u,v,w) \in (H^1(\R^3))^3:\langle(\frac{\partial U}{\partial x_1}, \frac{\partial V}{\partial x_1} ), (u,v)\rangle=0~~\hbox{and}~\langle\frac{\partial W}{\partial x_1}, w\rangle=0 \right\}.
$$

For any $R>0$, and any  $(\varphi,\psi,0) \in (C_0^\infty(B_R(0)))^3\cap \widehat{E}$ and be even in $y_2$ and $y_3$. Then $(\varphi_n(y),\psi_n(y)):=(\varphi(y-x^1),\psi(y-x^1)) \in C_0^\infty(B_{R}(x^1))\times C_0^\infty(B_{R}(x^1)),$ where we assume that $k_n$ is large enough such that $B_R(x^1)\subset \Omega_1$.
We may identity $(\varphi_n(y),\psi_n(y))$ as elements in $E$ by redefining the values outside $\Omega_1$ with the symmetry.
 Inserting $(\varphi_n(y),\psi_n(y),0)$ into \eqref{3.4} and taking the limit, we find that
\begin{equation}\label{2.10}
\begin{array}{rl}
&~~\ds\int_{\R^3}(\nabla u\nabla\varphi+u\varphi-\mu_1\phi_{U}u \varphi -2\mu_1\phi[U u] U\varphi-\beta\phi_{V}u \varphi -2\beta\phi[Uu]V\psi)\vspace{0.12cm}\\
&+\ds\int_{\R^3}(\nabla v\nabla \psi+v\psi-\mu_2\phi_{V}v\psi -2\mu_2\phi[Vv]V\psi-\beta\phi_{U}v\psi-2\beta\phi[Vv]U\varphi)=0.
\end{array}
\end{equation}
However, since $u$ and $v$ are even in $y_2$ and $y_3$, \eqref{2.10} holds for any function $(\varphi,\psi) \in C_0^\infty(B_R(0)) \times C_0^\infty(B_R(0))$, which is odd in $y_2$ or $y_3$. Therefore, \eqref{2.10} holds for any $(\varphi,\psi) \in C_0^\infty(B_R(0)) \times C_0^\infty(B_R(0))\cap \widehat{E}$.
By the density of $C_0^\infty(B_R(0)) \times C_0^\infty(B_R(0))$ in $H^1(\R^3) \times H^1(\R^3)$, we obtain that for $\forall (\varphi,\psi,0) \in \widehat{E}$,
\begin{equation}\label{2.11}\begin{array}{rl}
&~~\ds\int_{\R^3}(\nabla u\nabla\varphi+u\varphi-\mu_1\phi_{U}u \varphi -2\mu_1\phi[U u] U\varphi-\beta\phi_{V}u \varphi -2\beta\phi[Uu]V\psi)\\
&+\ds\int_{\R^3}(\nabla v\nabla \psi+v\psi-\mu_2\phi_{V}v\psi -2\mu_2\phi[Vv]V\psi-\beta\phi_{U}v\psi-2\beta\phi[Vv]U\varphi))=0.\\[2mm]
\end{array}
\end{equation}
Noting that $(U,V)=(\alpha w, \gamma w)$ and $w$ is a solution of \eqref{11.4}, we can show that \eqref{2.10} holds for
$(\varphi,\psi)=(\frac{\partial U}{\partial x_1},\frac{\partial V}{\partial x_1},0)$. Thus \eqref{2.10} is true for any
$(\varphi,\psi) \in H^1(\R^3) \times H^1(\R^3)$. Therefore, we have verified \eqref{2.9}.

Similarly, replacing $\Omega_1$ and $(\varphi,\psi,0)  $ by $\tilde{\Omega}_1$ and $(0,0,\xi) $ and repeating above process, we can see that  \eqref{11.20.1} is also true.

Noting  that $(U,V)$  and $W$ are  non-degenerate and we work in the space of functions
which are even in $y_2$ and $y_3$,  we get $(u,v)=c(\frac{\partial U}{\partial x_1},\frac{\partial V}{\partial x_1})$  and $w=C\frac{\partial W}{\partial x_1}$ for some constants $c$ and $C.$
From \eqref{11.20.2}, we can see $(u,v,w)=(0,0,0)$.

As a result,
$$
\ds\int_{B_R(x^1)}(\varphi_n^2+\psi_n^2)=o_n(1), \,\,\,\int_{B_R(y^1)}\xi_n^2=o_n(1),\,\forall R>0.
$$
Based on the properties of $w$, we can get that
$$
U_{r_n}, V_{r_n}\leq Ce^{-(1-\delta)|x-x^1|}, x\in \Omega_1;W_{\rho_n}\leq Ce^{-(1-\delta)|x-y^1|},x\in \tilde{\Omega}_1
$$
where $\delta\in (0,1)$ is a small constant.
So, following from the symmetry,  we have that
\begin{align}\label{2.12}
o_n(1)k_n
=&k_n-k_n\ds\int_{\Omega_1} \mu_{1}(\phi_{U_r}\varphi_n ^2 +2\phi[U_r\varphi_n]U_r\varphi_n
-k_n\ds\int_{\Omega_1} \mu_{2}(\phi_{V_r} \psi_n^2  +2\phi[V_r  \psi_n]V_r\psi_n\cr
&-k_n\ds\int_{\tilde{\Omega}_1} \mu_{3}(\phi_{W_\rho} \xi_n^2 +2\phi[W_\rho \xi_n]W_\rho\xi_n \cr
&+\beta_{12}\int_{\Omega_1}\phi_{V_r} \varphi_n^2 +2\phi[U_r \varphi_n]V_r\psi_n+\phi_{U_r} \psi_n^2+2\phi[U_r\varphi_n]V_r \psi_n\cr
&+\beta_{13}\int_{\Omega_1}\phi_{W_\rho}\varphi_n^2 +2\phi[U_r \varphi_n ]W_\rho\xi_n)+\phi_{U_r} \xi_n^2+2\phi[U_r\varphi_n]W_\rho \xi_n\cr
&+\beta_{23}\int_{\Omega_1}\phi_{W_\rho} \psi_n^2 +2\phi[V_r \psi_n]W_\rho \xi_n)
+\phi_{V_r}\xi_n^2 +2\phi[V_r\psi_n]W_\rho \xi_n)\cr
=&k_n+o_n(1)k_n-\beta_{13}k_n(\ds\int_{\Omega_1}
\phi_{W_\rho}\varphi_n^2+\ds\int_{\tilde{\Omega}_1}\phi_{U_r} \xi_n^2)-\beta_{23}k_n(\ds\int_{\Omega_1}
\phi_{W_\rho} \psi_n^2
+\ds\int_{\tilde{\Omega}_1}\phi_{V_r}\xi_n^2),
\end{align}
which leads to a contradiction for the case of $\beta_{13}\leq 0$ and $\beta_{23}\leq 0$.

If $\beta_{13}\leq 0$ and $\beta_{23}>0$, then it follows from \eqref{2.12} that there exists a constant $\beta_0>0$ such that
$$o_n(1)k_n\geq k_n-\beta_{23}k_n(\ds\int_{\Omega_1}
\phi_{W_\rho} \psi_n^2
+\ds\int_{\tilde{\Omega}_1}\phi_{V_r}\xi_n^2)+o_n(1)k_n$$
is impossible for the case of  $\beta_{13}\leq 0$ and $\beta_0>\beta_{23}>0$.
Similarly, if  $\beta_{13}>0$ and $\beta_{23}\leq 0$, then it follows from \eqref{2.12} that there exists a constant $\beta_0>0$ such that
$$o_n(1)k_n\geq k_n-\beta_{13}k_n(\ds\int_{\Omega_1}
\phi_{W_\rho}\varphi_n^2+\ds\int_{\tilde{\Omega}_1}\phi_{U_r} \xi_n^2)$$
is impossible for the case of  $\beta_0>\beta_{13}>0$ and $\beta_{23}\leq 0$.

In a word, there exists a $\beta_0>0$ such that when  $\beta_{13}\chi(\beta_{13})+\beta_{23}\chi(\beta_{23})<\beta_0$,
\eqref{2.12} leads to a contradiction.
So the proof is completed.
\end{proof}

\begin{lemma}\label{lemma2.4}
For any $(\varphi,\psi,\xi)\in E,$
we have
\begin{align*}
\|R^{(i)}(\varphi,\psi,\xi)\|=O(\|(\varphi,\psi,\xi)\| ^{3-i}
+ \|(\varphi,\psi,\xi)\|^{4-i}),
\end{align*}
\end{lemma}
\begin{proof}
By direct calculation, we have, for any $(u_1,v_1,w_1),(u_2,v_2,w_2)\in E$
\begin{align}
 &|R(\varphi,\psi,\xi)|\cr
=&|\ds\int_{\R^3}(\frac{\mu_{1}}{4}\phi_{\varphi}|\varphi|^2+\mu_{1}\phi_{\varphi}U_r\varphi
+\frac{\mu_{2}}{4}\phi_{\psi}|\psi|^2+\mu_{2}\phi_{\psi}V_r\psi
+ \frac{\mu_{3}}{4}\phi_{\xi}|\xi|^2+\mu_{3}\phi_{\xi}W_\rho\xi)
\cr
&+\frac{\beta_{12}}{2}\ds\int_{\R^3} (\phi_{\varphi}|\psi|^2+2\phi_{\psi}U_r\varphi
+2\phi_{\tilde{\varphi}}V_r\psi)
+\frac{\beta_{13}}{2}\ds\int_{\R^3} (\phi_{\varphi}|\xi|^2+2\phi_{\xi}U_r\varphi
+2\phi_{\varphi}W_\rho\xi)\cr
&+\frac{\beta_{23}}{2}\ds\int_{\R^3} (\phi_{\psi}|\xi|^2+2\phi_{\xi}V_r\psi
+2\phi_{\psi}W_\rho\xi)|\cr
\leq& C (|\varphi|_{\frac{12}{5}}+|\psi|_{\frac{12}{5}}+|\xi|_{\frac{12}{5}})^3+C(|\varphi|_{\frac{12}{5}}+|\psi|_{\frac{12}{5}}+|\xi|_{\frac{12}{5}})^4\vspace{0.12cm}\cr
\leq &C ||(\varphi,\psi,\xi)|| ^3+C ||(\varphi,\psi,\xi)|| ^4.
\end{align}

And by similar calculations, we get that
$$
\begin{array}{rl}
&|\langle R^\prime(\varphi,\psi,\xi),(u_1,v_1,w_1)\rangle |\leq C\big( ||(\varphi,\psi,\xi)\| ^2+ \|(\varphi,\psi,\xi)\| ^3\big)\|(u_1,v_1,w_1)\|,\\[2mm]
\end{array}$$
and
$$\begin{array}{rl}
|\langle R^{\prime\prime}(\varphi,\psi,\xi)(u_1,v_1,w_1),(u_2,v_2,w_2)\rangle | \leq C( \|(\varphi,\psi,\xi)\| + \|(\varphi,\psi,\xi)\| ^2)\|(u_1,v_1,w_1)\| \|(u_2,v_2,w_2)\|.\\[2mm]
\end{array}$$
So we complete the proof.
\end{proof}

\begin{lemma}\label{lm3.3}
For $k$ large enough, we have that
$$
\begin{array}{rl}
\|\ell\|&=O\Big(\ds\frac{k}{r^{\min\{m_1, m_2\}}}+\ds\frac{k}{\rho^{m_3}}+k \sum\limits_{j=1}^k\frac{1}{|x^1-y^j|}\Big)=O(1)\Big(\frac{1}{k}\Big)^\frac{2m-1}{1-m}\Big(\frac{1}{\ln k}\Big)^\frac{m}{1-m},
\end{array}
$$
where $m=\min\{m_1, m_2\}=m_3$.
%%计算耦合部分时利用$r$和$r$的同阶数代入计算。
\end{lemma}

\begin{proof}
Similar to (3.13) and (3.14) of \cite {hpwz}, we can see that
\begin{align}\label{3.21}
\ds\int_{\R^3}(V_1(x)-1)U_r\varphi\leq Ck \frac{1}{r^{m_1}}||\varphi||,\,\,\,\,\ds\int_{\R^3}(V_2(x)-1)V_r\psi\leq Ck \frac{1}{r^{m_2}}||\varphi||
\end{align}
\begin{align*}
\ds\int_{\R^3}(V_3(x)-\lambda)W_\rho\xi|
 \leq Ck \frac{1}{\rho^{m_3}}||\xi||
 \leq  Ck (\frac{1}{r^{m_1}}+ \frac{1}{r^{m_2}}+\frac{1}{\rho^{m_3}})||(\varphi,\psi,\xi)||,
\end{align*}
Also, we have
\begin{align}
\Bigl|-\mu_{1}\ds\int_{\R^3}\Big(\phi_{U_r}U_r -\sum\limits_{j=1}^{k}\phi_{U_{x^j }}U_{x^j }\Big)\varphi\Bigr|=O(\frac{k}{r^{m_1}})||(\varphi,\psi)||
\end{align}
\begin{align}
\Bigl|-\beta_{12}\ds\int_{\R^3} (\phi_{V_r}U_r -\sum\limits_{j=1}^{k}\phi_{V_{x^j }}U_{x^j } )\tilde{\varphi}\Bigr|=O(\frac{k}{r^{m_1}})||(\varphi,\psi)||
\end{align}
Similarly, we have
\begin{align}
&\Bigl|-\mu_{2}\ds\int_{\R^3}\Big(\phi_{V_r}V_r -\sum\limits_{j=1}^{k}\phi_{V_{x^j }}V_{x^j }\Big)\psi
-\beta_{12}\ds\int_{\R^3}\Big(\phi_{U_r}V_r-\sum\limits_{j=1}^{k}\phi_{U_{x^j }}V_{x^j }\Big)\psi\Bigr|\cr
=&O(\frac{k}{r^{m_2}})||(\varphi,\psi)||.
\end{align}
and
\begin{align*}
\Bigl|-\mu_{3}\ds\int_{\R^3}\Big(\phi_{W_\rho}W_\rho -\sum\limits_{j=1}^{k}\phi_{W_{y^j }}W_{y^j }\Big){\xi}\Bigr|=O(\frac{k}{\rho^{m_3}})||{\xi}||
\end{align*}
\begin{align}
\Bigl|-\beta_{13}\ds\int_{\R^3} (\phi_{W_\rho}U_r-\sum\limits_{j=1}^{k}\phi_{W_{y^j }}U_{x^j } ){\varphi}\Bigr|=kO(\sum\limits_{j=1}^k\frac{1}{|x^1-y^j|})||({\varphi},{\psi})||
\end{align}
\begin{align}
\Bigl|-\beta_{23}\ds\int_{\R^3}\Big(\phi_{W_\rho}V_r -\sum\limits_{j=1}^{k}\phi_{W_{y^j }}V_{x^j }\Big){\psi}\Bigr|
=kO(\sum\limits_{j=1}^k\frac{1}{|x^1-y^j|})||({\varphi},{\psi})||
\end{align}
similarly,we have
\begin{align}
\Bigl|-\beta_{13}\ds\int_{\R^3}\Big(\phi_{U_r}W_\rho-\sum\limits_{j=1}^{k}\phi_{U_{x^j }}W_{y^j }\Big){\xi}\Bigr|
=kO(\sum\limits_{j=1}^k\frac{1}{|x^1-y^j|})||{\xi}||
\end{align}
\begin{align}\label{3.22}
\Bigl|-\beta_{23}\ds\int_{\R^3}\Big(\phi_{V_r} W_\rho-\sum\limits_{j=1}^{k}\phi_{V_{x^j }}W_{y^j }\Big){\xi}\Bigr|
=kO(\sum\limits_{j=1}^k\frac{1}{|x^1-y^j|})||{\xi}||
\end{align}
Combining \eqref{3.21}-\eqref{3.22}, we complete the proof.
\end{proof}
Now we are in the position to prove proposition 2.1.

{\bf Proof of Proposition 2.1}
From the definition of $\ell({\varphi},{\psi}, {\xi})$, we know that $\ell({\varphi},{\psi}, {\xi})$ is a bounded
linear functional in ${E}$. Thus it follows
from Reisz  Representation Theorem that there is an $\ell^\prime \in E$ such
that
$$\ell({\varphi},{\psi}, {\xi})=\langle\ell^\prime,({\varphi},{\psi}, {\xi})\rangle.$$
So finding a critical point of $J({\varphi},{\psi}, {\xi})$ is equivalent to solving
\begin{equation}\label{3.7}
\ell^\prime+B({\varphi},{\psi}, {\xi})+R^\prime({\varphi},{\psi}, {\xi})=0.
\end{equation}

By Lemma \ref{lm3.2}, $B$ is invertible. Hence \eqref{3.7} can be written as
$$({\varphi},{\psi}, {\xi})={A}({\varphi},{\psi}, {\xi}):
=-B^{-1}\ell^\prime-\tilde{B}^{-1}R^\prime({\varphi},{\psi}, {\xi}).$$

We choose a small constant $0<\tau_2< \frac{1}{4}$ and set
$$\begin{array}{rl}
\tilde{S}=\bigg\{({\varphi},{\psi}, {\xi}) \in \tilde{E}:&\|({\varphi},{\psi}, {\xi})\|
\leq \Big(\ds\frac{1}{k}\Big)^\frac{2m-1}{1-m}\Big(\ds\frac{1}{\ln k}\Big)^\frac{(1-\tau_2)m}{1-m}\bigg\}.\\
\end{array}$$
For $k$ sufficiently large, we have
$$\begin{array}{rl}
 \|{A}({\varphi},{\psi}, {\xi})\|
&\leq C\|{\ell_k}\| +C\|{R}^\prime({\varphi},{\psi}, {\xi})\|\vspace{0.12cm}\\
&\leq C\Big(\ds\frac{1}{k}\Big)^\frac{2m-1}{1-m}\Big(\ds\frac{1}{\ln k}\Big)^\frac{m}{1-m}
+C\Big(\ds\frac{1}{k}\Big)^\frac{4m-2}{1-m}\Big(\ds\frac{1}{\ln k}\Big)^\frac{2(1-\tau_2)m}{1-m}\vspace{0.12cm}\\
&\leq\Big(\ds\frac{1}{k}\Big)^\frac{2m-1}{1-m}\Big(\ds\frac{1}{\ln k}\Big)^\frac{(1-\tau_2)m}{1-m},
\quad\forall ({\varphi},{\psi}) \in \tilde{S},
\end{array}$$
which implies that $A$ is a map from $\tilde{S}$ to $\tilde{S}$.

On the other hand, for $k$ sufficiently large, we get
$$
\begin{array}{rl}
&|{A}({\varphi}_1,{\psi}_1, {\xi}_1)-{A}({\varphi}_2,{\psi}_2, {\xi}_2)|\\[2mm]
&\leq C|{R}^\prime({\varphi}_1,{\psi}_1, {\xi}_1)-{R}^\prime({\varphi}_2,{\psi}_2, {\xi}_2)|\\[2mm]
&\leq C\|{R}^{\prime\prime}(\lambda({\varphi}_1,{\psi}_1, {\xi}_1)+(1-\lambda)({\varphi}_2,{\psi}_2, {\xi}_2))\|
\|({\varphi}_1,{\psi}_1, {\xi}_1)-({\varphi}_2,{\psi}_2, {\xi}_2))\|
\\[2mm]
&\leq C\big[(\|({\varphi}_1,{\psi}_1, {\xi}_1)\|
+\|({\varphi}_2,{\psi}_2, {\xi}_2)\|
)
+(\|({\varphi}_1,{\psi}_1, {\xi}_1)\|^2
+\|({\varphi}_2,{\psi}_2, {\xi}_2)\|^2
\big]\\[2mm]
&\quad\times\|({\varphi}_1,{\psi}_1, {\xi}_1)-({\varphi}_2,{\psi}_2, {\xi}_2))\|\\[2mm]
&\leq \ds\frac{1}{2}\|({\varphi}_1,{\psi}_1, {\xi}_1)-({\varphi}_2,{\psi}_2,{\xi}_2))\|.
\end{array}
$$
Thus for $k$ sufficiently large, ${A}$ is a contraction map. Therefore we have proved that when $k$ is sufficiently large,
${A}$ is a contraction map from $\tilde{S}$ to $\tilde{S}$. So the result follows from the  contraction mapping theorem. This completes the proof.\\

\section{Proof of Theorem \ref{Th3}}
Now we are ready to prove Theorem \ref{Th3}. Let $({\varphi}(r,\rho),{\psi}(r,\rho), {\xi}(r,\rho))$ be the map obtained in Proposition \ref{pro3.4}.
Define
$$
F(r,\rho)=I ({U}_r+{\varphi}(r,\rho),{V}_r +{\psi}(r,\rho), W_\rho+ {\xi}(r,\rho)),\,\,\, (r,\rho) \in S_k \times S_k.
$$
We can check that for $k$ sufficiently large, if $(r, \rho)$ is a critical point of $F(r, \rho )$, then $({U}_r+{\varphi}(r,\rho),{V}_r +{\psi}(r,\rho), W_\rho+ {\xi}(r,\rho)$ is a critical point of $I$.\\

Now we are ready to prove Theorem \ref{Th3}.
$$
\begin{array}{rl}
I(U_r,V_r, W_\rho )&=k \Big[[A_0+\tilde{A}+a_1B_1\ds\frac{1}{r^{m_1}}+ a_2B_2 \ds\frac{1}{r^{m_2}}+a_3B_3\ds\frac{1}{\rho^{m_3}}-D_0\ds\frac{k\ln k}{\pi r}-D_1\ds\frac{k\ln k}{\pi \rho}
\\
&-D_2\sum\limits_{j=1}^k\frac{1}{|x^1-y^j|}+o( \frac{1}{(k\ln k)^\frac{m}{1-m}})\Big]\\[2mm]
\end{array}
$$
To ensure Lemma \ref{lm3.2} to be true, we may assume that $\beta\in (-\infty, \beta_0),$ where $\beta_0$ is given out in Lemma \ref{lm3.2}.
 From Lemmas \ref{lemma2.4} , \ref{lm3.3}, and Propositions \ref{pro3.4} , \ref{proA5}, we have
\begin{equation}\label{F1}
\begin{array}{rl}
F(r, \rho)&=k \Big[[A_0+\tilde{A}+a_1B_1\ds\frac{1}{r^{m_1}}+ a_2B_2 \ds\frac{1}{r^{m_2}}+a_3B_3\ds\frac{1}{\rho^{m_3}}-D_0\ds\frac{k\ln k}{\pi r}-D_1\ds\frac{k\ln k}{\pi \rho}
\\
&-D_2\sum\limits_{j=1}^k\frac{1}{|x^1-y^j|}+o( \frac{1}{(k\ln k)^\frac{m}{1-m}})\Big].
\end{array}
\end{equation}

 We set
 $$I:=\Bigl\{(x,y)\in \R^+\times \R^+: y=x\Bigr\},$$
$$ B_\delta\big(d_1,d_2\big):=\Bigl\{(x,y)\in \R\times \R: |(x,y)-(d_1,d_2)|\leq \delta\Bigr\},$$
$$ \partial B_\delta\big(d_1,d_2\big):=\Bigl\{(x,y)\in \R\times \R: |(x,y)-(d_1,d_2)|=\delta\Bigr\},$$
\begin{equation}\label{F2}
\begin{array}{rl}
\hat{F}(x,y):&=F(x(k \ln k)^\frac{1}{1-m},y(k \ln k)^\frac{1}{1-m})
\vspace{0.12cm}\\&=k\Bigg[A_0+\tilde{A}+\Big(\ds\frac{1}{k \ln k}\Big)^\frac{m}{1-m}
\Big(a_1B_1\ds\frac{1}{x^{m_1}}+ a_2B_2 \ds\frac{1}{x^{m_2}}+a_3B_3\ds\frac{1}{y^{m_3}}-D_0\ds\frac{1}{x}-D_1\ds\frac{1}{y}\vspace{0.12cm}\\
&\quad\quad-D_2\pi\frac{1}{k \ln k}\sum\limits_{j=1}^k  \frac{1}{\sqrt{x^2+y^2-2xy\cos\frac{|2j-1|\pi}{k}}} +o(1)\Big)\Bigg],
\end{array}
\end{equation}
\begin{equation}\label{f1}
\begin{array}{rl}
f(x,y):&=
a_1B_1\ds\frac{1}{x^{m_1}}+ a_2B_2 \ds\frac{1}{x^{m_2}}+a_3B_3\ds\frac{1}{y^{m_3}}-D_0\ds\frac{1}{x}-D_1\ds\frac{1}{y}\vspace{0.12cm}\\
&\quad\quad-D_2\pi\frac{1}{k \ln k}\sum\limits_{j=1}^k  \frac{1}{\sqrt{x^2+y^2-2xy\cos\frac{|2j-1|\pi}{k}}} , (x,y)\in \R^+\times \R^+,
\end{array}
\end{equation}
 $$
 f_1(x):=\left\{\begin{array}{ll}&a_1B_1\ds\frac{1}{x^{m_1}} -D_0\ds\frac{1}{x},~~\hbox{for}~m_1<m_2,\\
& a_2B_2 \ds\frac{1}{x^{m_2}} -D_0\ds\frac{1}{x},~~\hbox{for}~m_2<m_1,\\
&(a_1B_1 + a_2B_2) \ds\frac{1}{x^{m_2}} -D_0\ds\frac{1}{x},~~\hbox{for}~m_1=m_2,\\
\end{array}\right.
 $$
 and
 $$
 f_2(x):= a_3B_3\ds\frac{1}{x^{m_3}} -D_1\ds\frac{1}{x}.
 $$

It is easy to check that
$$
\sup\limits_{x\in (0, +\infty)}f_1(x)=f_1(d_1)=\left\{\begin{array}{ll}&(1-m)(a_1B_1)^\frac{1}{1-m}\Big(\frac{m}{D_0}\Big)^\frac{m}{1-m},~~\hbox{for}~m_1<m_2,\\
&(1-m)(a_2B_2)^\frac{1}{1-m}\Big(\frac{m}{D_0}\Big)^\frac{m}{1-m},~~\hbox{for}~m_2<m_1,\\
&(1-m)(a_1B_1+a_2B_2)^\frac{1}{1-m}\Big(\frac{m}{D_0}\Big)^\frac{m}{1-m},~~\hbox{for}~m_1=m_2,\\
\end{array}\right.
$$
and
$$
\sup\limits_{x\in (0, +\infty)}f_2(x)=f_2(d_2)=(1-m)(a_3B_3)^\frac{1}{1-m}\Big(\frac{m}{D_1}\Big)^\frac{m}{1-m},
$$
where $d_1:=\left\{\begin{array}{ll}&(\frac{D_0}{a_1B_1 m})^\frac{1}{1-m},~~\hbox{for}~m_1<m_2,\\[2mm]
&(\frac{D_0}{a_2B_2m})^\frac{1}{1-m},~~\hbox{for}~m_2<m_1, \\[2mm]
&(\frac{D_0}{(a_1B_1+a_2B_2)m})^\frac{1}{1-m},~~\hbox{for}~m_1=m_2, \\[2mm]
\end{array}\right.$ and $d_2:=(\frac{D_1}{a_3b_3m})^\frac{1}{1-m}.$
By the definitions of ${B}_1, {B}_2, {D}_0$ and ${D}_1$, we can see that
$$(\frac{d_1}{d_2})^{1-m}=\left\{\begin{array}{ll}&\frac{a_3(\alpha^2+\gamma^2)}{a_1\alpha^2\lambda^{\frac{1}{2}}},~~\hbox{for}~m_1<m_2, \\[2mm]
&\frac{a_3(\alpha^2+\gamma^2)}{a_2\gamma^2\lambda^{\frac{1}{2}}},~~\hbox{for}~m_2<m_1,\\[2mm]
&\frac{a_3(\alpha^2+\gamma^2)}{(a_1\alpha^2+a_2\gamma^2)\lambda^{\frac{1}{2}}},~~\hbox{for}~m_1=m_2,\\[2mm]
\end{array}\right.$$

\textbf{Case 1:}~~  $\beta_{13}\chi(\beta_{13})+\beta_{23}\chi(\beta_{23})<\beta_0$  $d_1\neq d_2$

Since $d_1\neq d_2$, we can find a small constant ${\delta_1}>0$ such that $B_{\delta_1}\big((d_1,d_2)\big)\subset
\R^+\times \R^+$,  and $B_{\delta_1}\big((d_1,d_2)\big)$ and the line $y=x$ do not have any intersection points, which implies that there exist $0<R_1<R_2$ such that
$$R_1<|x-y|<R_2~\hbox{for~all~}(x,y)\in \partial B_{\delta_1}\big((d_1,d_2). $$
%$d_3:=d_1-\delta, d_4:=d_1+\delta, d_5：=d_2- \delta$
%and
%$d_6：=d_2+ \delta$
%are four different positive constants, where $\delta\in (0,\min\{\frac{|d_1-d_2|}{5}, \frac{d_1}{2}, \frac{d_2}{2}\})$
%is a small constant, determined later.
Due to the continuity of $f_1(x)+f_2(y)$, there exists a $(x_0,y_0)\in \partial    B_{\delta_1}\big((d_1,d_2)\big)$ such that
$$f_1(x_0)+f_2(y_0)=\sup\limits_{(x,y)\in \partial B_{\delta_1}\big((d_1,d_2)\big)}f_1(x)+f_2(y).$$

Using Lemma \ref{lmb1}, we can see that, for any fixed $\beta\in (-\infty, \beta_0)$, there exists $k_0>0$ such that for any $k>k_0$,
$$\begin{array}{ll}
f(d_1, d_2)\vspace{0.12cm}
&\geq f_1(d_1)+f_2(d_2)
-\ds\frac{|\beta| }{\mu_1\mu_2}\frac{C_w}{2} \frac{1}{ \ln k}\frac{1}{|d_1- d_2 |}\vspace{0.12cm}\\
&> f_1(x_0)+f_2(y_0)
+\ds\frac{|\beta| }{\mu_1\mu_2}\frac{C_w}{2} \frac{1}{ \ln k}\frac{1}{R_1}\vspace{0.12cm}\\
&\geq f_1(x)+f_2(y)
+\ds\frac{|\beta| }{\mu_1\mu_2}\frac{C_w}{2} \frac{1}{ \ln k}\frac{1}{|x- y|}\vspace{0.12cm}\\
&\geq f(x,y),~~~~\forall (x,y)\in \partial B_{\delta_1}\big((d_1,d_2)\big).
\end{array}
$$
Therefore,  for any fixed $\beta\in (-\infty, \beta_0)$, there exists $K_0>0$ such that for any $k>K_0$,
$$\hat{F}(d_1, d_2)>\hat{F}(x,y),~~~~\forall (x,y)\in \partial B_{\delta_1}\big((d_1,d_2)\big).$$
Thus,  for any fixed $\beta\in (-\infty, \beta_0)$, when $k$ is large enough,  there exists $(\tilde{x}_0, \tilde{y}_0)\in B_{\delta_1}\big((d_1,d_2)\big)\setminus \partial B_{\delta_1}\big((d_1,d_2)\big)$ such that
$$\hat{F}(\tilde{x}_0, \tilde{y}_0)=\max\limits_{(x , y )\in B_{\delta_1}\big((d_1,d_2)\big)}\hat{F}(x,y),$$
which implies that
$\tilde{F}(r, \rho)$ has a local maximum point $    (r_0, \rho_0)$ in $B^k_{\delta_1}\big((d_1,d_2)\big) $ and $(r_0, \rho_0)$  is an interior point of  $B^k_{\delta_1}\big((d_1,d_2)\big)$, where $B^k_R\big((d_1,d_2)\big):=\{(x,y)\in \R^+\times \R^+: |(x,y)-(d_1(k \ln k)^\frac{1}{1-m},d_2(k \ln k)^\frac{1}{1-m})|\leq R (k \ln k)^\frac{1}{1-m}\}.$
So $(r_0, \rho_0)$ is a critical point of $F$ and  $(\tilde{U}_{r_0}+\tilde{\varphi}(r_0,\rho_0),\tilde{V}_{\rho_0}+\tilde{\psi}(r_0, \rho_0))$ is a solution of \eqref{1.1}.

\textbf{Case 2:}~~ $\beta_{13}\chi(\beta_{13})+\beta_{23}\chi(\beta_{23})<\beta_0$, $0\leq D_2<\frac{d_1}{2}\min\{f_1(d_1), f_2(d_2)\}$ and $a = b $, i.e,  $d_1=d_2.$

Following from Lemma \ref{lmb1}, we can see that, for $x=y$,
$$\sum\limits_{j=1}^k  \frac{1}{\sqrt{x^2+y^2-2xy\cos\frac{|2j-1|\pi}{k}}}= \frac{2}{\pi x}(1+o_k(1))k \ln k.$$
Since $0\leq D_2<\frac{d_1}{2}\min\{f_1(d_1), f_2(d_2)\}$, we can find a $\epsilon_1>0$ such that $\min\{f_1(d_1), f_2(d_2)\}-\frac{2+\epsilon_1}{ d_1}D_2>0.$
Therefore, we obtain that, $k$ large enough,
$$\begin{array}{ll}f(d_1,d_2)&\geq f_1(d_1)+ f_2(d_2)-D_2\frac{2+\epsilon_1}{  d_1}
>\max\{f_1(d_1),~ f_2(d_2)\}\vspace{0.12cm}\\
&= \max\{\sup\limits_{x\in (0, +\infty)}f_1(x),~ \sup\limits_{x\in (0, +\infty)}f_2(x)\}\vspace{0.12cm}\\
&=\limsup\limits_{|(x,y)|\to +\infty}f(x,y)\\
\end{array}$$
At the same time, we have that, as $|(x,y)|\to 0^+,$
$$
\begin{array}{ll}
f(x,y)
&\leq a\ds\tilde{B}\frac{1}{x^m}+b\tilde{C} \ds\frac{1}{y^m}- \tilde{D}_1\ds\frac{1}{x} -\tilde{D}_2\ds\frac{1}{y}
\to -\infty,
\end{array}
$$
which, together with $\limsup\limits_{|(x,y)|\to +\infty}f(x,y)<f(d_1,~d_2)$,  implies  that there exists $0<{\delta_2}<\frac{1}{4}$ such that, ${\delta_2}<|(d_1,~d_2)|<\frac{1}{{\delta_2}}$ and  for any $(x,y)\in \R^+\times \R^+$  with $|(x,y)|\leq {\delta_2}$ or $|(x,y)|\geq \frac{1}{{\delta_2}}$,
 $$
 f(x,y)<f(d_1,~d_2).
 $$

On the other hand, we have that, for any $x\in (\frac{{\delta_2}}{2}, \frac{2}{{\delta_2}})$,
$$
f(x,y)\leq  C+b\ds\tilde{C} \frac{1}{y^m}  -\tilde{D}_2\frac{1}{y}
\to -\infty, ~~\hbox{as}~y\to 0^+,
$$
 and for any $y\in (\frac{{\delta_2}}{2}, \frac{2}{{\delta_2}})$,
$$f(x,y)\leq C+a\tilde{B}\frac{1}{x^m} - \tilde{D}_1\frac{1}{x}
 \to -\infty, ~~\hbox{as}~x\to 0^+.$$

 Therefore, there exists a $(x_0, y_0)\in \R^+\times\R^+$ such that
$f(x_0, y_0)=\max\limits_{(x,y)\in \R^+\times\R^+}f(x,y),$  which implies that we can find a $\delta_3>0$ such that
$f(x,y )<f(x_0, y_0),~\forall~(x,y)\in \partial B_{\delta_3}((x_0,y_0)).$ Proceeding as Case 1, we can get  a critical point $(r_0, \rho_0)$   of $F$, which means that  $(\tilde{U}_{r_0}+\tilde{\varphi}(r_0,\rho_0),\tilde{V}_{\rho_0}+\tilde{\psi}(r_0, \rho_0))$ is a solution of \eqref{1.1}.

\textbf{Case 3:}~~  $D_2<0, D_0+D_1+2D_2>0$ and $a= b$, i.e,  $d_1=d_2.$
Note that
\begin{equation}\label{bu20.1.1}
\begin{array}{ll}
\limsup\limits_{|(x,y)|\to +\infty}f(x,y)
&= \max\{\sup\limits_{x\in (0, +\infty)}f_1(x),~ \sup\limits_{x\in (0, +\infty)}f_2(x)\}\vspace{0.12cm}\\
&=\max\{f_1(d_1),~ f_2(d_2)\}\vspace{0.12cm}\\
&<f_1(d_1)+f_2(d_2)
<f(d_1,d_2).
\end{array}
\end{equation}

Since $D_0+D_1+2D_2>0$,  there exists a $\epsilon_0 >0$ such that $D_0+D_1\frac{1}{1+\epsilon}+2D_2\frac{1+|\epsilon|}{\min\{1, 1+\epsilon\}}>\frac{{\delta_4}}{3}$ and $D_0\frac{1}{1+\epsilon}+D_1+2D_2\frac{1+|\epsilon|}{\min\{1, 1+\epsilon\}}>\frac{{\delta_4}}{3}$ for any $\epsilon\in [-\epsilon_0, \epsilon_0],$ where $\delta_4:=\min\{D_0+D_1+2D_2,~D_0,~D_1\},$ and which implies that , for $l\in \R^+$ with  $|l-1|\leq \epsilon_0$
$$
\begin{array}{ll}
&D_0+D_1\ds\frac{1}{l}+D_2\pi\ds\frac{1}{k \ln k}\sum\limits_{j=1}^k\frac{1}{\sqrt{1+l^2-2l\cos \frac{(2j-1)\pi}{k}}}\vspace{0.12cm}\\
&\geq
D_0+D_1\ds\frac{1}{l}+2D_2\frac{1+|l-1|}{\min\{1, l\}}
>\ds\frac{{\delta_4}}{3}
\end{array}
$$
and
$$
\begin{array}{ll}
&D_0\ds\frac{1}{l}+D_1+D_2\pi\ds\frac{1}{k \ln k}\sum\limits_{j=1}^k\frac{1}{\sqrt{1+l^2-2l\cos \frac{(2j-1)\pi}{k}}}\vspace{0.12cm}\\
&\geq
D_0\ds\frac{1}{l}+D_1+2D_2\frac{1+|l-1|}{\min\{1, l\}}
>\ds\frac{{\delta_4}}{3}
\end{array}
$$

 On the other hand,
there exists a $k_0>0$ such that when $k\geq k_0$ and $l\in \R^+$ with  $|l-1|\geq \epsilon_0$
$$
\begin{array}{ll}
&D_0+D_1\ds\frac{1}{l}+D_2\pi\ds\frac{1}{k \ln k}\sum\limits_{j=1}^k\frac{1}{\sqrt{1+l^2-2l\cos \frac{(2j-1)\pi}{k}}}\vspace{0.12cm}\\
&\geq
D_0+D_1\ds\frac{1}{l}+D_2\pi\ds\frac{1}{k \ln k}\ds\frac{k}{|l-1|}\geq
D_0 +D_2\pi\ds\frac{1}{ \ln k_0}\ds\frac{1}{\epsilon_0}
>\ds\frac{{\delta_4}}{3}
\end{array}
$$
and
$$
\begin{array}{ll}
&D_0\ds\frac{1}{l}+D_1+D_2\pi\ds\frac{1}{k \ln k}\sum\limits_{j=1}^k\frac{1}{\sqrt{1+l^2-2l\cos \frac{(2j-1)\pi}{k}}}\vspace{0.12cm}\\
&\geq
D_0\ds\frac{1}{l}+D_1+D_2\pi\ds\frac{1}{k \ln k}\ds\frac{k}{|l-1|}\geq
D_1 +D_2\pi\ds\frac{1}{ \ln k_0}\ds\frac{1}{\epsilon_0}
>\ds\frac{{\delta_4}}{3}
\end{array}
$$
In a word, if $D_2<0$ and $D_0+D_1+2D_2>0$, then  there exists a $k_0>0$ such that when $k\geq k_0$,
$$\tilde{D}_1+\tilde{D}_2\frac{1}{l}+\frac{C_w\beta_{12}}{2\mu_1\mu_2}\frac{1}{k \ln k}\sum\limits_{j=1}^k\frac{1}{\sqrt{1+l^2-2l\cos \frac{(2j-1)\pi}{k}}}>\frac{{\delta_4}}{3},
\forall l\in \R^+$$
and
$$\tilde{D}_1\frac{1}{l}+\tilde{D}_2+\frac{C_w\beta_{12}}{2\mu_1\mu_2}\frac{1}{k \ln k}\sum\limits_{j=1}^k\frac{1}{\sqrt{1+l^2-2l\cos \frac{(2j-1)\pi}{k}}}>\frac{{\delta_4}}{3},
\forall l\in \R^+.$$

Therefore, when  $\frac{-\pi\mu_1\mu_2(\tilde{D}_1+\tilde{D}_2)}{C_w}<\beta<0$ and $k$ large enough, we have that $y=lx,1\leq l\in \R^+$,
$$
\begin{array}{ll}
&\lim\limits_{,x\to 0^+}f(x,y)\vspace{0.12cm}\\
&=\lim\limits_{x\to 0^+}\Bigg[\ds\frac{1}{x^m}\Big(a\tilde{B}+b\tilde{C}\frac{1}{l^m}\Big)-
\frac{1}{x}\Big(\tilde{D}_1+\tilde{D}_2\frac{1}{l}
+\ds\frac{C_w\beta_{12}}{2\mu_1\mu_2}\frac{1}{k \ln k}\sum\limits_{j=1}^k\frac{1}{\sqrt{1+l^2-2l\cos \frac{(2j-1)\pi}{k}}}\Big)\Bigg]\vspace{0.12cm}\\
&\leq\lim\limits_{x\to 0^+}\Big[\ds\frac{1}{x^m}(a\tilde{B}+b\tilde{C})-
\ds\frac{1}{x}\frac{{\delta_4}}{3}\Big]
=\lim\limits_{x\to 0^+}\ds\frac{1}{x^m}\Big[(a\tilde{B}+b\tilde{C})-
\ds\frac{1}{x^{1-m}}\frac{{\delta_4}}{3}\Big]\to -\infty
\end{array}
$$
and
$$
\begin{array}{ll}
&\lim\limits_{y\to 0^+}f(x,y)\vspace{0.12cm}\\
&=\lim\limits_{y\to 0^+}\Bigg[\frac{1}{y^m}\Big(a\tilde{B}\frac{1}{l^m}+b\tilde{C}\Big)-
\ds\frac{1}{y}\Big(\tilde{D}_1\frac{1}{l}+\tilde{D}_2
+\ds\frac{C_w\beta_{12}}{2\mu_1\mu_2}\frac{1}{k \ln k}\sum\limits_{j=1}^k\frac{1}{\sqrt{l^2 +1-2l\cos \frac{(2j-1)\pi}{k}}}\Big)\Bigg]\vspace{0.12cm}\\
&\leq\lim\limits_{y\to 0^+}\Big[\ds\frac{1}{y^m}(a\tilde{B}+b\tilde{C})-
\ds\frac{1}{y}\frac{{\delta_4}}{3}\Big]
=\lim\limits_{y\to 0^+}\ds\frac{1}{y^m}\Big[(a\tilde{B}+b\tilde{C})-
\ds\frac{1}{y^{1-m}}\frac{{\delta_4}}{3}\Big]\to -\infty,
\end{array}
$$
which, together with \eqref{bu20.1.1}, implies that there exits a ${\delta_5}>0$ such that for any $(x,y)\in \R^+\times \R^+$ with $|(x,y)|\geq \frac{1}{{\delta_5}}$ or
$x\leq {\delta_5}$ or $y\leq {\delta_5}$,
$$f(x,y)<f(d_1,d_2).$$
So
  there exists a $(x_0,y_0)\in D_{{\delta_5}}:=\{(x,y)\in \R^+\times \R^+: |(x,y)|<\frac{1}{{\delta_5}},~
x>{\delta_5},~ y> {\delta_5}\}$
such that
$$f(x_0, y_0)=\max\limits_{(x,y)\in D_{{\delta_5}}}f(x,y).$$
Proceeding as Case 2, we can get  a critical point $(r_0, \rho_0)$   of $F$, which means that  $({U}_{r_0}+{\varphi}(r_0,\rho_0),{V}_{\rho_0}+{\psi}(r_0, \rho_0))$ is a solution of \eqref{1.1}.

\appendix

\section{Energy estimate}
\def\theequation{A.\arabic{equation}}\makeatother
\setcounter{equation}{0}

In this section, we will obtain some energy estimates of the approximate solutions.
Recall that
$$ x^j:=\Big(r\cos\frac{2(j-1)\pi}{k},~r\sin\frac{2(j-1)\pi}{k},~x_3\Big),~j=1,2,\cdots,k,$$
$$y^j:=\Big(\rho\cos\frac{(2j-1)\pi}{k},~\rho\sin\frac{(2j-1)\pi}{k},~x_3\Big),j=1,2,\cdots,k,$$
Also, we define
$$U_r(x)=\sum\limits_{j=1}^{k}\big(sgn(\pm)\big)^jU_{x^j },
V_r(x)=\sum\limits_{j=1}^{k}\big(sgn(\pm)\big)^jV_{x^j },
W_\rho(x)=\sum\limits_{j=1}^{k}\big(sgn(\pm)\big)^jW_{y^j },
$$

$$
\begin{array}{rl}
I(u_1,u_2,u_3)&=\ds\frac{1}{2}\sum\limits_{j=1}^3\ds\int_{\R^3}(|\nabla u_j|^2 +V_j(x)u^2_j)~dx -\frac{1}{4}\sum\limits_{j=1}^3\mu_{j}\ds\int_{\R^3}\phi_{u_j}u_j^2~dx-\frac{1}{4}\sum\limits_{i\neq j}\beta_{ij}\ds\int_{\R^3}\phi_{u_j}u_i^2~dx\\
&=I(u_1,u_2,0)+I(0,0,u_3)-\frac{1}{2}\beta_{13}\ds\int_{\R^3}\phi_{u_1}u_3^2-\frac{1}{2}\beta_{23}\int_{\R^3}\phi_{u_2}u_3^2,\\
\end{array}
$$
and
$$g(x,y):=\sum\limits_{j=1}^k\frac{1}{\sqrt{x^2+y^2-2xy\cos\frac{(2j-1)\pi}{k}}}, ~~x, y\in \R^+.$$

\begin{lemma}(\cite{hpwz})\label{lmb1}
There holds that
$$g(x,y) =\frac{2}{\pi x}(1+o_k(1))k \ln k,~~~\hbox{for }~~x=y$$
and, for $x\neq y$,
$$\frac{k}{2\min\{x,y\} +|x-y|}< g(x,y)<\min\Big\{\frac{k }{|x-y|},~ \frac{2}{\pi \min\{x, y\}}(1+o_k(1))k \ln k \Big\}.$$
\end{lemma}

\begin{proposition}(\cite{hpwz}Proposition A.1 )\label{harproA1}
We obtain that
\begin{equation}\label{h1.1}
\ds\int_{\R^3}\phi_{w_{x^j }}w_{x^j }^2
=\ds\int_{\R^3}\phi_{w_{y^j }}w_{y^j }^2
= \int_{\R^3}\phi_ww^2,
\end{equation}
\begin{equation}\label{h1.2}
\ds\int_{\R^3}\phi_{w_{x^j }}w_{x^j }w_{x^i }=
\ds\int_{\R^3}\phi_{w_{y^j }}w_{y^j }w_{y^i }
=O\Big(  e^{-{(1-\tau)r\sin\frac{|i-j|}{k}\pi}}\Big ),
\end{equation}
\begin{equation}\label{h1.3}
\ds\int_{\R^3}\phi_{w_{x^j }}w_{x^i }^2=
\ds\int_{\R^3}\phi_{w_{x^j }}w_{x^i }^2
=C_w \frac{1}{|x^i-x^j|}+O\Big(  \frac{1}{|x^i-x^j|^{1+\tau}}\Big),~i\neq j,\end{equation}

\begin{equation}\label{bh1.3}
\ds\int_{\R^3}\phi_{w_{x^j }}W_{y^i }^2=
\ds\int_{\R^3}\phi_{W_{y^i }}w_{x^j }^2
=C_w \frac{\lambda^\frac{1}{2}}{\mu_{3}}\frac{1}{|y^i-x^j|}+O\Big(  \frac{1}{|y^i-x^j|^{1+\tau}}\Big),
\end{equation}
and
$$
\int_{\R^3}\phi_{W_{y^i}}W_{y^j}^2=C_w\frac{\lambda}{\mu^2_{3}}\frac{1}{|y^i-y^j|}+o(\frac{1}{|y^i-y^j|}),~~i\neq j,
$$
where $C_w$ is a positive constant, depending on $w$.
\end{proposition}

\begin{proposition}(\cite{hpwz} Proposition A.2)\label{proA1}
Assume that $(A_1)$ holds. Then we get the following energy estimate: for any $j=1,2, \cdots, k$,
$$
\frac{1}{2}\ds\int_{\R^3}(V_1(x)-1)U_{x^j }^2)
= a_1B_1 \frac{1}{r^{m_1}}
+O\Big( \frac{1}{r^{m_1+\min\{1,\theta_1\}}}\Big),
$$

$$\frac{1}{2}\ds\int_{\R^3}(V_2(x)-1)V_{x^j }^2)
= a_2B_2 \frac{1}{r^{m_2}}
+O\Big( \frac{1}{r^{m_2+\min\{1,\theta_2\}}}\Big),$$
and
$$\frac{1}{2}\ds\int_{\R^3}(V_3(x)-\lambda)W_{y^j }^2)
= a_3B_3 \frac{1}{\rho^{m_3}}
+O\Big( \frac{1}{\rho^{m_3+\min\{1,\theta_3\}}}\Big),$$
%$$\int_{\R^3}\phi_{U_{x^i}}W_{y^j}^2=C_w\frac{\alpha^2\lambda^\frac{1}{2}}{\beta_{33}} \frac{1}{|x^i-y^j|}+o(\frac{1}{|x^i-y^j|}),~~i,j=1,2, \cdots, k,$$
%$$\int_{\R^3}\phi_{V_{x^i}}W_{y^j}^2=C_w\frac{\gamma^2\lambda^\frac{1}{2}}{\beta_{33}} \frac{1}{|x^i-y^j|}+o(\frac{1}{|x^i-y^j|}),~~i,j=1,2, \cdots, k,$$
where $a_i, m_i, \theta_i$ are  given in  $(A_1)$,
 $B_1=\ds\frac{1}{2}\alpha^2\ds\int_{\R^3}w^2~dx$,  $B_2=\ds\frac{1}{2}\gamma^2\ds\int_{\R^3}w^2~dx$ and $B_3=\ds\frac{1}{2}\frac{\lambda^\frac{1}{2}}{\mu_{3}}\ds\int_{\R^3}w^2~dx$.
\end{proposition}

\begin{proposition}(\cite{hpwz} Proposition A.3)\label{proA2.1}
Assume that $(A_1)$ holds. Then we have that
$$
\begin{array}
{rl}I(U_{x^i },V_{x^i })&=A_0+a_1B_1 \ds\frac{1}{r^{m_1}}+ a_2B_2 \ds\frac{1}{r^{m_2}}
+O\Big( \ds\frac{1}{r^{m_1+\min\{1,\theta_1\}}}+ \ds\frac{1}{r^{m_2+\min\{1,\theta_2\}}}\Big),
\end{array}
$$
where $A_0:=\ds\frac{1}{4}(\mu_{1}\alpha^4+\mu_{2}\gamma^4+2\mu_{1}\mu_{2}\alpha^2\gamma^2)\int \phi_w w^2~$, and  $a_i, m_i, \theta_i$ are given in $(A_1)$, and $B_1$ and $B_2$ have been defined in Proposition \ref{proA1}.
\end{proposition}

\begin{proposition}(\cite{hpwz} Proposition A.4)\label{proA2}
Assume that $(A_1)$ holds. Then we have that
$$
\begin{array}
{rl}I (U_r,V_r)&=k \Big[A_0+a_1B_1\ds\frac{1}{r^{m_1}}+ a_2B_2 \ds\frac{1}{r^{m_2}}-D_0\ds\frac{k\ln k}{ r}
+o\Big(\ds \frac{1}{r^{\min\{m_1, m_2\}}}+\ds\frac{k\ln k}{\pi r} \Big)\Big],
\end{array}
$$
where $D_0:=\ds\frac{1}{4}(\mu_{1}\alpha^4+\mu_{2}\gamma^4+2\mu_{1}\mu_{2}\alpha^2\gamma^2)\frac{C_w}{\pi}$ and $a_i, m_i, \theta_i$ are given in $(A_1)$, and $A_0, B_1$ and $B_2$ have been defined in Proposition \ref{proA1}.
\end{proposition}

Similar to Proposition \ref{proA2}, we can get the following Proposition.

\begin{proposition}\label{proA4}
Assume that $(A_1)$ holds. Then we have that
$$
\begin{array}{rl}
I(0,0, W_\rho)&
= k\Big[\tilde{A}
 +a_3B_3\ds\frac{1}{\rho^{m_3}}-D_1 \ds\frac{k\ln k}{ \rho}
+o\big(\ds\frac{1}{\rho^{m_3}}+\ds\frac{k\ln k}{  \rho}\big)\Big],
 \end{array}
 $$
where $\tilde{A}=\ds\frac{\mu_{3}}{4}\ds\int_{\R^3}\phi_{W}W^2~dx=\ds\frac{\lambda^\frac{3}{2}}{4\mu_{3}}\ds\int_{\R^3}\phi_{w}w^2~dx$, $D_1=\frac{1}{4}\frac{\lambda}{\mu_{3}}\frac{C_w}{\pi}$,  $a_3, m_3$ are  given in $(A_1)$,
  and $B_3$ has been defined in Proposition \ref{proA1}.
\end{proposition}

\begin{proposition}\label{proA5}
Assume that $(V)$   holds. Then
$$
\begin{array}{rl}
I(U_r,V_r, W_\rho )&=k \Big[[A_0+\tilde{A}+a_1B_1\ds\frac{1}{r^{m_1}}+ a_2B_2 \ds\frac{1}{r^{m_2}}+a_3B_3\ds\frac{1}{\rho^{m_3}}-D_0\ds\frac{k\ln k}{  r}-D_1\ds\frac{k\ln k}{  \rho}
\\
&-D_2\pi \sum\limits_{j=1}^k\frac{1}{|x^1-y^j|}+o( \frac{1}{(k\ln k)^\frac{m}{1-m}})\Big]\\[2mm]
\end{array}
$$
where $D_2:=\Big(\frac{\beta_{13}\alpha^2+\beta_{23}\gamma^2}{2}\Big)\frac{\lambda^\frac{1}{2}}{\mu_{3}}\frac{C_w}{\pi} $ and $m:=\min\{m_1, m_2, m_3\}$.
\end{proposition}
\begin{proof}
Following from  Proposition \ref{proA1}, we obtain that
$$\begin{array}{ll}
\ds\int_{\R^3}\phi_{U_r}W_\rho^2&=\sum\limits_{i,j=1}^k\ds\int_{\R^3}\phi_{U_{x^i}}W_{y^j}^2+O\big(\sum\limits_{i,j,l=1,j\neq l}^k(\int_{\R^3}\phi_{U_{x^i}}W_{y^j} W_{y^l}+\int_{\R^3}\phi_{W_{y^i}}U_{x^j} U_{x^l})\big)\\
&=k\sum\limits_{j=1}^k\ds\int_{\R^3}\phi_{U_{x^1}}W_{y^j}^2+O\Big(  e^{-{(1-\tau)r\sin\frac{|i-j|}{k}\pi}}\Big ),\\
&=C_w\frac{\alpha^2\lambda^\frac{1}{2}}{\mu_{3}}k\sum\limits_{j=1}^k \frac{1}{|x^1- y^j|}+o(k\sum\limits_{j=1}^k \frac{1}{|x^1- y^j|})+O\Big( k\sum\limits_{j=1}^k e^{-{(1-\tau)r\sin\frac{|1-j|}{k}\pi}}\Big ),
\end{array}
$$
and
$$\begin{array}{ll}
\ds\int_{\R^3}\phi_{V_r}W_\rho^2&=\sum\limits_{i,j=1}^k\ds\int_{\R^3}\phi_{V_{x^i}}W_{y^j}^2+O\big(\sum\limits_{i,j,l=1,j\neq l}^k(\int_{\R^3}\phi_{V_{x^i}}W_{y^j} W_{y^l}+\int_{\R^3}\phi_{W_{y^i}}V_{x^j} V_{x^l})\big)\\
&=k\sum\limits_{j=1}^k\ds\int_{\R^3}\phi_{V_{x^1}}W_{y^j}^2+O\Big(  e^{-{(1-\tau)r\sin\frac{|i-j|}{k}\pi}}\Big ),\\
&=C_w\frac{\gamma^2\lambda^\frac{1}{2}}{\mu_{3}}k\sum\limits_{j=1}^k \frac{1}{|x^1- y^j|}+o(k\sum\limits_{j=1}^k \frac{1}{|x^1- y^j|})+O\Big( k\sum\limits_{j=1}^k e^{-{(1-\tau)r\sin\frac{|i-j|}{k}\pi}}\Big ).
\end{array}
$$
Therefore,
$$
\begin{array}{rl}
&I(U_r,V_r, W_\rho )\\
&=I(U_r, V_r, 0)+I(0, 0, W_\rho)-\frac{1}{2}\beta_{13}\ds\int_{\R^3}\phi_{U_r}W_\rho^2-\frac{1}{2}\beta_{23}\ds\int_{\R^3}\phi_{V_r}W_\rho^2\\
&=k \Big[A_0+\tilde{A}+a_1B_1\ds\frac{1}{r^{m_1}}+ a_2B_2 \ds\frac{1}{r^{m_2}}+a_3B_3\ds\frac{1}{\rho^{m_3}}-D_0\ds\frac{k\ln k}{\pi r}-D_1\ds\frac{k\ln k}{\pi \rho}
\\
&-D_2\sum\limits_{j=1}^k\frac{1}{|x^1- y^j|}+o\Big(\ds \frac{1}{r^{\min\{m_1,m_2\}}}+\ds\frac{k\ln k}{\pi r} \Big)\Big]+o\big(\ds\frac{1}{\rho^{m_3}}+\ds\frac{k\ln k}{\pi \rho}\big)\Big]+o(\sum\limits_{j=1}^k\frac{1}{|x^1- y^j|})\Big]\\
&=k \Big[[A_0+\tilde{A}+a_1B_1\ds\frac{1}{r^{m_1}}+ a_2B_2 \ds\frac{1}{r^{m_2}}+a_3B_3\ds\frac{1}{\rho^{m_3}}-D_0\ds\frac{k\ln k}{\pi r}-D_1\ds\frac{k\ln k}{\pi \rho}
\\
&-D_2\sum\limits_{j=1}^k\frac{1}{|x^1-y^j|}+o( \frac{1}{(k\ln k)^\frac{m}{1-m}})\Big]
\end{array}
$$
This completes the proof.
%$$
%\begin{array}{ll}
%\sum\limits_{i,j=1}^k\frac{1}{|x^i-y^j|}
%&=k\sum\limits_{j=1}^k \frac{1}{|x^1-y^j|}\\[2mm]
%&=\left\{
%\begin{array}{ll}
%&2k\sum\limits_{j=1}^{\frac{k}{2}} \frac{1}{\sqrt{\rho^2+r^2-2\rho r\cos\frac{(2j-1)\pi}{k}}},~~k~\hbox{is~even}\\[2mm]
%&2k\sum\limits_{j=1}^{\frac{k-1}{2}} \frac{1}{\sqrt{\rho^2+r^2-2\rho r\cos\frac{(2j-1)\pi}{k}}}+\frac{1}{\sqrt{\rho^2+r^2-2\rho r\cos\pi}},~~k~\hbox{is~odd}\\[2mm]
%\end{array}
%\right.
%\end{array}
%$$

\end{proof}
\noindent{\bf Acknowledgements.}

\noindent This paper  was  supported by   Guangxi Natural Science Foundation(2025GXNSFFA069011),  National Natural Science Foundation of China (12061012, 12461022),  and  Guangxi Province Talent Project (Financial support from Science and Technology Department).
The authors would like to thank Professor Chunhua Wang from Central China Normal University for her helpful discussions.

\end{document}